\documentclass[12pt,a4paper]{article}
\usepackage{amssymb,amsmath,amsthm}
\usepackage{filecontents}
\usepackage[utf8x]{inputenc}
\usepackage{subcaption}
\usepackage{graphicx}
\usepackage[numbers,square]{natbib}
\usepackage{sectsty}
\sectionfont{\normalsize}
\subsectionfont{\normalsize}
\subsubsectionfont{\normalsize}
\paragraphfont{\normalsize}
\newtheorem{thm}{Theorem}
\newtheorem*{thm*}{Theorem}
\newtheorem{lem}[thm]{Lemma}

\newtheorem{cor}[thm]{Corollary}
\newtheorem{clm}[thm]{Claim}
\newcommand{\llll}[1] {\left #1}
\newcommand{\rrrr}[1] {\right #1}
\newcommand{\wwww}[1]{\widetilde #1}
\newcommand{\oooo}[1]{\overline #1}
\newcommand{\dddd}[2]{\dfrac{#1}{#2}}
\newcommand{\pppp}{\partial}
\newcommand{\aaaa}{\alpha}
\newcommand{\tttt}{\tau}

\newcommand{\bbbb}{\beta}
\begin{document}
\nocite{*}
\title{{\bf \normalsize NUMERICAL APPROXIMATIONS FOR FRACTIONAL DIFFERENTIAL EQUATIONS}}
\author{Yuri Dimitrov\\
Department of Applied Mathematics and Statistics \\
University of Rousse  \\
8 Studentska str.\\
Rousse  7017, Bulgaria\\
\texttt{ymdimitrov@uni-ruse.bg}}
\maketitle
\begin{abstract}
The Gr\"unwald and shifted Gr\"unwald formulas for the  function $y(x)-y(b)$ are first order approximations  for the Caputo fractional derivative of the function $y(x)$ with lower limit at the point $b$.  We obtain  second and third order approximations for the Gr\"unwald and shifted Gr\"unwald formulas with weighted averages of Caputo derivatives  when  sufficient number of derivatives of the function $y(x) $ are equal to zero at  $b$, using  the  estimate for the error of the shifted Gr\"unwald formulas.  We use the approximations  to determine implicit difference   approximations for the  sub-diffusion equation which have second order accuracy  with respect to the  space and time  variables, and second and  third order numerical approximations for ordinary fractional differential equations.  
\end{abstract}
\section{Introduction}
\label{Intro} Fractional derivatives are an effective tool for modeling  diffusion processes in complex systems.  Mathematical models with partial fractional differential equations have been used to describe complex processes in physics, biology, chemistry and economics [10-20]. The time fractional diffusion equation  is a parabolic partial fractional differential equation obtained from the heat-diffusion equation by replacing the time derivative with a fractional derivative of order $\alpha$.
\begin{equation}\label{fdeqn}
\dfrac{\partial^\alpha u(x,t)}{\partial t^\alpha}=\dfrac{\partial^2 u(x,t)}{\partial x^2}+G(x,t).
\end{equation}
When $0<\alpha<1$ the equation is called fractional sub-diffusion  equation and it is a model of a slow diffusion process. When the order of the fractional deriative is between one and two the  equation is called fractional super-diffusion  equation. There is a growing need to design efficient algorithms for numerical solution of partial fractional differential equations. Finite-difference approximations for the  heat-diffusion $(\aaaa=2)$ and time-fractional diffusion  equations have been studied [38-50] for their importance in practical applications as well as for evaluation of their performance.  

A common way to approximate the Caputo derivative of order $\aaaa$, when $0<\aaaa<1$, is approximation  \eqref{CaputoApproximation}. The finite difference approximation for the fractional sub-diffusion equation which uses  approximation \eqref{CaputoApproximation} for the time fractional derivative and central difference approximation for the second derivative with respect to $x$   has accuracy $O(\tau^{2-\alpha}+h^2)$, where  $h$ and $\tau$ are the step sizes of the discretizations  with respect to the space and time variables $x$ and $t$.  

Tadjeran et al. \cite{TadjeranMeerschaertScheffer2006} use the estimate for the error for the Gr\"unwald formula \eqref{estimate} to design  an algorithm for a second order  numerical approximation of the solution of the space fractional diffusion equation of order $\aaaa$, when $1<\aaaa<2$.
$$\dddd{\partial u(x,t)}{\partial t}=d(x)\dddd{\partial^\aaaa u(x,t)}{\partial x^\aaaa}+G(x,t).$$
The algorithm uses a Crank-Nicholson approximation with respect to the time variable $t$ and an extrapolation with respect to the space variable $x$. 

Ding and Li \cite{DingLi2013} compute a numerical solution of the fractional diffusion-wave equation with reaction term
$$\dddd{\partial^\aaaa u(x,t)}{\partial t^\aaaa}=K_\aaaa\dddd{\partial^2 u(x,t)}{\partial x^2}-C_\aaaa u(x,t)+G(x,t),$$
 where $K_\aaaa>0$ and $C_\aaaa>0$ are the diffusion and reaction coefficients, using a compact difference approximation with an accuracy $O\llll(\tau^2+h^4\rrrr)$. 

Gorenflo \cite{Gorenflo1996} showed that the shifted Gr\"unwald formula 
$h^{-\aaaa} \Delta_{h,\aaaa/2}^\aaaa y(x)$ is a second order approximation for the fractional derivative  
$$y^{\aaaa}(x)=h^{-\aaaa}\Delta_{h,\aaaa/2}^\aaaa y(x)+O\llll(h^2\rrrr)$$
when the transition of $y(x)$ to zero is sufficiently smooth at the lower limit of fractional differentiation. The smoothness condition requires that $y(b)=0$. This approximation is a special case of Theorem 1(i), when $p=\aaaa/2$.

There is a significant interest in designing efficient numerical solutions for ordinary and partial fractional differential equations, which stems from the possibility to use fractional derivatives to explain complex processes in nature and social sciences and their relation to integer order differential equations. While approximations of fractional derivatives  with accuracy $O\llll( h^{2-\aaaa}\rrrr)$ have been studied extensively, new algorithms with second and higher order accuracy [27-37] have been  proposed and successfully applied for numerical solution of ordinary and partial fractional differential equations.

In this paper we construct a second order implicit difference approximation for the fractional sub-diffusion equation \eqref{fdeqn} and second and third order approximations for the ordinary fractional differential equation
\begin{equation} \label{eq2}
y^{(\aaaa)}(x)+y(x)=f(x)
\end{equation}
when equations  \eqref{fdeqn} and \eqref{eq2} have sufficiently smooth solutions. In section 3 we use the estimate for the error of the Gr\"unwald and shifted Gr\"unwald formulas \eqref{estimate} to obtain second and third order approximations for the Gr\"unwald and shifted Gr\"unwald formulas with weighted averages of Caputo derivatives on a uniform grid, when sufficient number of derivatives of the function $y(x)$ are  equal to zero at the point $b$. In section 4 we derive recurrence relations \eqref{Recurrence3},  \eqref{Recurrence4} and \eqref{3rdOrder} for  second and third  order approximations to the solution of ordinary fractional differential  equation \eqref{eq2} and we determine estimates for the Gr\"unwald weights. In section 5 we use  approximations \eqref{SecondOrder} and \eqref{SecondOrder2} to construct implicit  difference approximations  \eqref{FiniteDifferenceScheme} and \eqref{FiniteDifferenceScheme2} for the solution of the sub-diffusion equation \eqref{fdeqn} and we show that they have second order accuracy $O\llll(\tau^2+h^2\rrrr)$ with respect to the space and time variables.
\section{Preliminaries}
The fractional derivatives are generalizations of the integer order derivatives. Let $y(x)$ be a real-valued function defined for $x\geq b$. The  Riemann-Liouville and Caputo  fractional derivatives   of order $\aaaa$, when $0<\alpha<1$ are defined as
\begin{equation}\label{RLder}
D_{RL}^\alpha y(x)=\dfrac{1}{\Gamma (1-\alpha)}\dfrac{d}{dx}\int_b^x \dfrac{y(\xi)}{(x-\xi)^\alpha}d\xi,
\end{equation}
\begin{equation}\label{Cder}
D_x^\alpha y(x)=y^{(\alpha)}(x)=\dddd{d^\aaaa}{d x^\aaaa}y(x)=\dfrac{1}{\Gamma (1-\alpha)}\int_b^x \dfrac{y'(\xi)}{(x-\xi)^\alpha}d\xi.
\end{equation}
The Caputo fractional derivative of the constant function $1$ is zero and the Riemann-Liouville  derivative of $1$ is $(x-b)^{-\alpha}/\Gamma (1-\alpha)$. The Caputo derivative of the function $y(x)$ satisfies
$$D_x^\alpha y(x)=D_x^\alpha (y(x)-y(b)).$$
If a function $y(x)$ is Caputo differentiable of order $\alpha$ then it is differentiable in the sense of the definition of Riemann-Liouville derivative. The classes Caputo differentiable functions $C^\aaaa$,   when $0<\aaaa<1$, include the $C^1$ functions and are suitable for numerical computations for fractional differentiable equations. (For a strict definition of a fractional derivative we assume Lebesgue integration in the definitions of Caputo and Riemann-Liouville fractional derivatives.) The Caputo and Riemann-Liouville fractional derivatives satisfy \cite[p. 53]{Diethelm2010}
\begin{equation}\label{CaputoRL}
D_{RL}^\alpha y(x)=D_{C}^\alpha y(x)+\dfrac{y(b)}{\Gamma \left(1-\alpha \right) (x-b)^\alpha}.
\end{equation}
The Caputo and Riemann-Liouville fractional derivatives of the function $y(x)$ are equal when $y(b)=0$. 

The Miller-Ross sequential fractional derivative of order $\aaaa_1+\aaaa_2$ for the  Caputo derivative is defined as
$$y^{(\aaaa_1+\aaaa_2)}(x)=D_x^{\aaaa_1}D_x^{\aaaa_2}y(x).$$
In the special cases $\aaaa_1=1,\aaaa_1=2$ and $\aaaa_2=\aaaa$, when $0<\aaaa<1$
$$y^{(1+\aaaa)}(x)=\dfrac{d}{dx}y^{(\aaaa)}(x),\quad y^{(2+\aaaa)}(x)=\dfrac{d^2}{dx^2}y^{(\aaaa)}(x).$$  
The local behavior of a differentiable function is described with its Taylor series and Taylor polynomials. The  properties of a function $y(x)$ close to the lower limit $b$ can be described with its Caputo and Miller-Ross derivatives at the point $b$. The following theorem is a generalization of the Mean-Value Theorem for differentiable functions \cite{TrujilloRiveroBonilla1999}.
\begin{thm*} (Generalized Mean-Value Theorem) Let $y\in C^\aaaa [b,x]$. Then
$$y(x)=y(b)+\dddd{(x-b)^\aaaa}{\Gamma (\aaaa+1)} y^{(\aaaa)}(\xi_x)\quad (b\leq\xi_x\leq x).$$
\end{thm*}
Fractional Taylor series for Caputo and Miller-Ross derivatives \cite{TrujilloRiveroBonilla1999,JosipPecaricPeric2005} can be defined using an approach similar to the derivation of the classical Taylor series which involves only integer order derivatives. While the local properties of a function at the lower limit $b$ can be explained with its fractional derivatives their values at any other point $x$ depend on the values of the function on the interval $[b,x]$. We can observe a similarity between the Taylor series of the function $y(x)$ and relation \eqref{estimate} for its fractional derivatives and the shifted Gr\"unwald formulas.

Two important special functions in fractional calculus are the gamma  and Mittag-Leffler functions. The gamma function has  properties 
$$\Gamma(0)=1, \quad \Gamma (z+1)=z\Gamma (z).$$
When $n$ is a positive integer $\Gamma(n)=(n-1)!$. The one-parameter and two-parameter Mittag-Leffler functions are defined for $\aaaa>0$ as
$$E_\alpha (z)=\sum_{n=0}^\infty \dfrac{z^n}{\Gamma(\alpha n+1)}, \quad
E_{\alpha,\beta} (z)=\sum_{n=0}^\infty \dfrac{z^n}{\Gamma(\alpha n+\beta)}.$$
Some special cases of the one-parameter Mittag-Leffler function
$$E_1(-z)=e^{-z},\quad E_2(-z^2)=\cos z,\quad E_{\frac{1}{2}}(z)=e^{z^2}erfc(-z),$$
where $ercf (z)$ is the complimentary error function 
$$ercf (z)=\dfrac{2}{\sqrt \pi}\int_z^\infty e^{-t^2}dt.$$
The Mittag-Leffler functions appear in the solutions of ordinary and fractional differential equations \cite[chap. 1]{Podlubny1999}. The ordinary differential equation 
$$y''(x)=y(x),y(0)=1,y'(0)=1$$
has solution $y(x)=E_{2,1}(x^2)+x E_{2,2}(x^2)=\cosh x+\sinh x=e^x$. 
The fractional differential equation 
$$y^{(\alpha)}(x)=\lambda y(x), y(0)=1$$
has solution $y(x)=E_\alpha (\lambda x^\alpha)$. We can determine the analytical solutions of linear  ordinary and partial fractional differential equations using integral transforms. The following formulas for the Laplace transform of the derivatives of the Mittag-Leffler functions are often used to determine the analytical solutions of linear fractional differential equations \cite{Podlubny1999}.  
$$\mathcal{L}\{ t^{\aaaa k+\bbbb-1} E^{(k)}_{\aaaa,\bbbb}(\pm a t^\aaaa) \}(s)=
\dddd{k! s^{\aaaa-\bbbb}}{\llll( s^\aaaa \mp a\rrrr)^{k+1}}.$$
 Analytical solutions of ordinary and partial fractional differential equations can be found only for special cases of the equations and the initial and boundary conditions. We can determine numerical approximations for the solutions of a much larger class of  equations   which include nonlinear fractional differential equations. Approximations for the Caputo and Riemann-Liouville derivatives are obtained from the Gr\"unwald-Letnikov fractional derivative  and by approximating the fractional integral in the definition. 

 Let   $x_n=b+nh$  be a uniform grid on the $x$-axis starting from the point $b$, and $y_n=y(x_n)=y(b+n h)$, where $h>0$ is a small number.
The following approximation  of the Caputo derivative is derived from quadrature approximations of the fractional integral for $y'(x)$ in the definition of Caputo fractional derivative \cite{ZhuangLiu2006} on the interval $\llll[b,x_n\rrrr]$, by approximating the integrals on all subintervals of length $h$.
 \begin{equation} \label{CaputoApproximation}
y^{(\alpha)}_n  = \dfrac{1}{h^\alpha}\sum_{k=0}^{n-1} c_k^{(\alpha)} y_{n-k}+O(h^{2-\alpha}),
\end{equation}
where $c_0^{(\alpha)}=1/\Gamma (2-\alpha)$ and
$$c_k^{(\alpha)}=\dddd{(k+1)^{1-\alpha}-2k^{1-\alpha}+(k-1)^{1-\alpha}}{\Gamma(2-\alpha)}.$$
 When the function $y(x)$ has continuous second derivative, approximation \eqref{CaputoApproximation} has accuracy $O(h^{2-\alpha})$. The weights $c_k^{(\alpha)}$  satisfy
\begin{equation*}
c_0^{(\alpha)}>0,\quad c_1^{(\alpha)}<c_2^{(\alpha)}<\cdots<c_k^{(\alpha)}<\cdots<0,\quad \sum_{k= 0}^\infty c_k^{(\alpha)} = 0.
\end{equation*}

The Gr\"unwald-Letnikov  fractional derivative is closely related to  Caputo and  Riemann-Liouville derivatives
$$D_{GL}^\alpha y(x) = \lim_{\Delta x\downarrow 0} \dfrac{1}{\Delta x^\alpha} \sum_{n=0}^{\left[\frac{x-b}{\Delta x}\right]} (-1)^n\binom{\alpha}{n} y(x-n \Delta x).$$
The two most often used values of the lower limit of the fractional differentiation $b$ are zero and $-\infty$.
When the lower limit $b=-\infty$ the upper limit of the sum in the definition of Gr\"unwald-Letnikov derivative is $\infty$. The fractional binomial coefficients are defined  similarly to the integer binomial coefficients with the gamma function
$$\binom{\alpha}{n}=\dfrac{\Gamma(\alpha+1)}{\Gamma(n+1)\Gamma(\alpha-n+1)}=\dfrac{\alpha(\alpha-1)\cdots(\alpha-n+1)}{n!}.$$
The Riemann-Liouville, Caputo and Gr\"unwald-Letnikov  derivatives are equal when $y(b)=0$ and $y\in C^1[b,x]$  (\cite[p. 43]{Diethelm2010}). Let's denote by  $\Delta_h^\alpha y(x)$ and $\Delta_{h,p}^\alpha y(x)$ the {\it Gr\"unwald difference operator} and {\it shifted Gr\"unwald  difference operators} for the function $y(x)$.
\begin{equation*}\label{GrunwaldDO}
\Delta_h^\alpha y(x)= \sum_{n=0}^{N_{x,h}} (-1)^n\binom{\alpha}{n} y(x-n h),
\end{equation*}
$$\Delta_{h,p}^\alpha y(x)= \sum_{n=0}^{N_{x,h}} (-1)^n\binom{\alpha}{n} y(x-(n-p) h).$$
where $N_{x,h}=\left[\frac{x-b}{h}\right]$, and $h>0$ is a small number.
 The Gr\"unwald operator is a special case of the shifted Gr\"unwald operator when the shift value $p=0$.    When $y(b)=0$, we derive approximations for the Riemann-Liouville  and Caputo derivatives from the definition of Gr\"unwald-Letnikov derivative. 
\begin{equation*}
D_{x}^\alpha y(x)=D_{RL}^\alpha y(x)=D_{GL}^\alpha y(x)\approx h^{-\aaaa} \Delta_h^\aaaa y(x)\approx h^{-\aaaa} \Delta_{h,p}^\aaaa y(x).
\end{equation*}
 We  will call the approximations $h^{-\aaaa} \Delta_h^\aaaa y(x)$ and $h^{-\aaaa} \Delta_{h,p}^\aaaa y(x)$ of the fractional derivative of order $\aaaa$ -  {\it Gr\"unwald formula} and  {\it shifted Gr\"unwald formulas} of the function $y(x)$.   Let $ w_n^{(\alpha)}$ be the weights of the Gr\"unwald formulas.
$$ w_n^{(\alpha)}=(-1)^n\binom{\alpha}{n}.$$
When  $y(x)$ is a continuously-differentiable function  the  shifted Gr\"unwald formulas for the function $y(x)-y(b)$ are first-order approximations for the Caputo derivative of $y(x)$.  
\begin{equation}
y^{(\alpha)}(x)=\dddd{1}{h^\aaaa}\sum_{n=0}^{N_{x,h}} w_n^{(\alpha)} (y(x-(n-p)h)-y(b))+O\left(h\right).
\label{ShiftedGrunwaldApproximation}
\end{equation}
In Theorem 1(i)  we show that the shifted Gr\"unwald formulas of the function $y(x)-y(b)$ are second-order approximations for the Caputo fractional derivative at the point $x+\left(p-\alpha / 2\right)h$, when the function $y(x)$ is sufficiently differentiable on the interval $[b,x]$ and the values of its first and second derivatives are equal to zero at the point $b$.
\begin{equation}\label{2ndOrdApp}
y^{(\alpha )} \left(x+\left(p-\dfrac{\alpha}{2}\right) h\right)=
\dddd{1}{h^\aaaa}\sum_{n=0}^{N_{x,h}} w_n^{(\alpha)} (y(x-(n-p)h)-y(b))+O\left(h^2\right).
\end{equation}
 We derive the above formula in Theorem 1(i) from relation \eqref{estimate} for   the shifted Gr\"unwald formulas and the fractional derivatives of the function $y(x)$.  When the conditions of Theorem 1 are satisfied the Riemann-Liouville and Caputo derivatives are equal. The values of a function and the shifted Gr\"unwald formulas satisfy
$$\beta_1  y(x_1)+\beta_2 y(x_2)=  y(\beta_1 x_1+\beta_2 x_2)+O\left(h^2\right),$$
$$\beta_1 \Delta_{h,p}^\alpha y(x_1)+\beta_2 \Delta_{h,q}^\alpha y(x_2)=
\Delta_{h,\beta_1 p+\beta_2 q}^\alpha y(\beta_1 x_1+\beta_2 x_2)+O\left(h^2\right).$$
when $\beta_1+\beta_2=1$ and $y(x)$ is a sufficiently smooth function.
 An alternative way to obain \eqref{2ndOrdApp} is to apply the approximation for  average value of shifted Grunwald formulas to approximation (2.15) in Tian et al. \cite{TianZhouDeng2012}.
An important special case of the above formula is when $y(b)=0$ and $p=0$.
\begin{equation} \label{SecondOrder}
y^{(\alpha )} \left(x-\dfrac{\alpha  h}{2}\right)=\dddd{1}{h^\aaaa}\sum_{n=0}^{N_{x,h}} w_n^{(\alpha)} y(x-n h)+O\left(h^2\right).
\end{equation}
This approximation is closely related to formula (2.9) in \cite{Gorenflo1996} and is suitable for constructing second order weighted numerical approximations for fractional differential equations on a uniform grid. As a direct consequence of \eqref{SecondOrder} we obtain a second order approximation for the Gr\"unwald formula using average values of Caputo derivatives on consecutive nodes of a uniform grid.
\begin{equation} \label{SecondOrder2}
\dddd{1}{h^\aaaa}\sum_{n=0}^{N_{x,h}} w_n^{(\alpha)} y(x-n h)=\llll(\dfrac{\aaaa}{2}  \rrrr)y^{(\alpha )}_{n-1}+
\llll(1-\dfrac{\aaaa}{2} \rrrr)y^{(\alpha )}_{n}+O\left(h^2\right).
\end{equation}
In Corollary 5 we determine a third order approximation for the Gr\"unwald formula using a weighted average of three consecutive  values of the Caputo derivative on a uniform grid, for sufficiently differentiable functions $y(x)$ which satisfy the conditions of Theorem 1(ii). 
\begin{align}  \nonumber \label{ThirdOrder}
\dddd{1}{h^\aaaa}\sum_{n=0}^{N_{x,h}} w_n^{(\alpha)} y(x-n h)=&\llll(\dddd{a^2}{8}-\dddd{5a}{24}\rrrr)y^{(\aaaa)}_{n-2}+\llll(\dddd{11a}{12}-\dddd{a^2}{4}\rrrr)y^{(\aaaa)}_{n-1}+\\
&\llll(1-\dddd{17a}{24}+\dddd{a^2}{8}\rrrr)y^{(\aaaa)}_{n}+O\llll( h^3\rrrr).
\end{align}
Approximations \eqref{SecondOrder},\eqref{SecondOrder2} and \eqref{ThirdOrder} are suitable for numerical computations for fractional differential equations on a uniform grid, when the solutions are sufficiently differentiable functions. In section 4 we determine second and third order approximations for ordinary differential equation \eqref{eq2}, and in section 5 we construct stable difference approximations for the fractional sub-diffusion equation which have second order accuracy with respect to the space and time variables. 
 The  Gr\"unwald weights $ w_n^{(\alpha)}$ are computed recursively with $ w_0^{(\alpha)}=1,  w_1^{(\alpha)}=-\alpha$ and  
$$ w_n^{(\alpha)}=\left( 1-\dfrac{\alpha+1}{n}\right)w_{n-1}^{(\alpha)}.$$
The numbers $ w_n^{(\alpha)}$ are the coefficients of the binomial series
$$(1-z)^\alpha=\sum_{n=0}^\infty  w_n^{(\alpha)} z^n.$$
 When $\alpha$ is a positive integer the sum is finite, and the binomial 
series converges  at the point $z=1$ when $0<\alpha <1$. 
The weights $ w_n^{(\alpha)}$ have the following properties.
\begin{equation}\label{GrunwaldProperties}
w_0^{(\alpha)}>0,\quad  w_1^{(\alpha)}<w_2^{(\alpha)}<\cdots<w_n^{(\alpha)}<\cdots<0,\quad \sum_{n=0}^\infty w_n^{(\alpha)}=0.
\end{equation}
When the lower limit of fractional differentiation  $b\neq \infty$, the upper limit of the sum  is finite   and $\sum_{n=0}^{N_{x,h}} w_n^{(\alpha)}>0$.
The shifted Gr\"unwald formulas for $y(x)$ are first order approximations for the Riemann-Liouville derivative of the function $y(x)$. The approximation error  can be represented as a sum of higher order Riemann-Liouville fractional derivatives \eqref{estimate}. This  estimate is obtained in Tadjeran et al. \cite{TadjeranMeerschaertScheffer2006} when the order $\aaaa$ of the Riemann-Liouville derivative is between one and two using Fourier transform of the shifted Gr\"unwald formulas. The estimate for the error of the shifted Gr\"unwald formulas is generalized in Hejazi et al. \cite{HejaziMoroneyLiu2014} for arbitrary positive $\aaaa$ and $p$.
\begin{thm*} Let $\alpha$ and $p$ be positive numbers,  and suppose that $y\in C^{[\alpha]+n+2}(\mathbb{R})$ and all derivatives of $y$ up to order $[\alpha]+n+2$ belong to $L^1(\mathbb{R})$. Then if $b=-\infty$, there exist constants $c_l$ independent of $h,y,x$ such that
\begin{equation}\label{estimate}
h^{-\alpha}\Delta_{h,p}^\alpha y(x)=D_{RL}^\aaaa y(x)+\sum_{l=1}^{n-1}c_l h^l D_{RL}^{\aaaa+l}y(x)+O(h^n).  
\end{equation}
\end{thm*}
The numbers $c_l$ are the coefficients of the series expansion of the function
$$\omega_{\aaaa,p}(z)=\llll( \dfrac{1-e^{-z}}{z}\rrrr)^\aaaa e^{pz}.$$
If a function $y(x)$ is defined  on a finite interval $\mathcal{I}$ we can extend it to the real line by setting $y(x)=0$ when  $x\notin \mathcal{I}$. In this way the Caputo derivative of the extended function is equal to zero when $x\notin \mathcal{I}$. In the next section we use \eqref{estimate} to detrmine second and third order approximations for the Gr\"unwald and shifted Gr\"unwald formulas with weighted averages of  Caputo derivatives on consecutive points of a uniform grid. 
\section{Approximations for Gr\"unwald and shifted Gr\"unwald formulas}
The definitions of Riemann-Liouville and Caputo  derivatives are the two most commonly used definitions for fractional derivatives.  
  The  Caputo and Riemann-Liouville  fractional derivatives of order $n+\alpha$, where $n$ is a positive integer and $0<\aaaa<1$ are defined as
$$D_x^{n+\alpha} y(x)=\dfrac{1}{\Gamma (1-\alpha)}\int_b^x \dfrac{y^{(n+1)}(\xi)}{(x-\xi)^\alpha}d\xi,$$
$$D_{RL}^{n+\alpha} y(x)=\dfrac{1}{\Gamma (1-\alpha)}\dddd{d^{n+1}}{dx^{n+1}
}\int_b^x \dfrac{y(\xi)}{(x-\xi)^\alpha}d\xi.$$
When $y(x)$ is sufficiently differentiable function on the interval $[b,x]$, the Riemann-Liouville and Caputo derivatives are related as \cite[p. 53]{Diethelm2010}
\begin{equation}\label{RL_C}
D_{RL}^{n+\aaaa} y(x)=D_{x}^{n+\aaaa} y(x)+
\sum_{k=0}^n \dfrac{y^{(k)}(b)}{\Gamma (k-\aaaa-1)}(x-b)^{k-\aaaa-n}.
\end{equation}
We can compute the value of the Riemann-Liouville derivative of the function $y(x)$ from the value of the Caputo derivative and  its integer order derivatives at the lower limit $b$. One disadvantage of of the Riemann-Liouville derivative is that it has a singularity at the  lower limit $b$. When $\aaaa>1$ the singularity is non-integrable.  The class of Caputo differentiable functions of order $n+\aaaa$ includes the functions with $n+1$ continuous derivatives. In section 2 we discussed properties of the Caputo derivatives at the lower limit $b$, which are related to the properties of the integer order derivatives.  
We often prefer to study fractional differential equations with the Caputo derivative because its properties make it an attractive fractional derivative for numerical solution of fractional differential equations.
 Let $h$ be the step size of a uniform grid on the $x$-axis staring from the lower limit of fractional differentiation $b$ and
$$x_n=b+nh,\quad y_n=y(x_n)=y(b+n h).$$
In Theorem 1 we use the estimate for the error of the shifted Gr\"unwald formulas \eqref{estimate} and relation   \eqref{RL_C} for the Riemann-Liouville and Caputo derivatives   to obtain second and third order approximations for the the shifted Gr\"unwald formulas using fractional order Caputo and Miller-Ross derivatives. 
\begin{thm} (Approximations for the shifted Gr\"unwald formulas)

 (i) Let $y(b)=y'(b)=y''(b)=0$ and $y\in C^{4}[b,x]$. Then
\begin{equation}\label{t1_1}
h^{-\alpha}\Delta_{h,p}^\alpha y(x)=y^{(\alpha )} \left(x+\left(p-\dfrac{\alpha}{2}\right) h\right)+O\left(h^2\right);
\end{equation}

 (ii) Let $y(0)=y'(0)=y''(0)=y'''(b)=0$ and $y\in C^{5}[b,x]$. Then
$$h^{-\alpha}\Delta_{h,p}^\alpha y(x)=y^{(\alpha )} \left(x+\left(p-\dfrac{\alpha}{2}\right) h\right)+\dddd{\aaaa}{24}h^2y^{(2+\aaaa)}(x)+O\left(h^3\right).$$
\end{thm}
\begin{proof} The function $\omega_{\aaaa,p}(z)$ has power series expansion
$$\omega_{\aaaa,p}(z)=\llll( \dfrac{1-e^{-z}}{z}\rrrr)^\aaaa e^{pz}=c_0+c_1z+c_2z^2+\cdots,$$
where $c_0=1, c_1=p-\aaaa/2$ and \cite{ChenDeng2013_2}
$$c_2=\dddd{1}{24} (12p^2-12\aaaa p+\aaaa+3\aaaa^2)=\dddd{\aaaa}{24}+
\dddd{1}{2}\llll(p^2-\aaaa p+\dddd{\aaaa^2}{4}\rrrr)=\dddd{\aaaa}{24}+\dddd{1}{2}\llll(p-\dddd{\aaaa}{2}\rrrr)^2.$$
From \eqref{estimate} and $n=2$ we obtain
$$h^{-\alpha}\Delta_{h,p}^\alpha y(x)= D_{RL}^\aaaa y(x)+\left( p-\dfrac{\alpha}{2}\right)h D_{RL}^{\alpha+1} y(x)+O\left(h^2\right).  $$
Let $0<\bbbb<1$. The Caputo and Riemann-Liouville fractional derivatives of order $3+\bbbb$ satisfy \eqref{RL_C} with $n=3$. The Riemann-Liouville  derivative $D_{RL}^{3+ \bbbb} y(x)$  has non-integrable singularities of orders $1+\bbbb,2+\bbbb$ and $3+\bbbb$ at the lower limit $b$, with coefficients $y(b),y'(b)$ and $y''(b)$. The Caputo derivative of $y(x)$ of order $3+\bbbb$ and the function $(x-b)^{-\bbbb}$  are integrable on a finite interval. Therefore  $D_{RL}^{3+\bbbb} y(x)\in L^1[b,x]$ when $y(b)=y'(b)=y''(b)=0$. Similarly,  $D_{RL}^{4+\bbbb} y(x)\in L^1[b,x]$ when $y(b)=y'(b)=y''(b)=y'''(0)=0$. 

The Riemann-Liouville and Caputo derivatives of order $\aaaa$ for the functions $y(x)$ which satisfy the conditions of Theorem 1 are equal 
$$D_{RL}^\aaaa y(x)=y^{(\aaaa)}(x)$$
and we can represent $D_{RL}^{\alpha+1} y(x)$ with the Caputo derivative $y^{(\aaaa)}(x)$ as
$$D_{RL}^{\alpha+1} y(x)=\dfrac{1}{\Gamma (1-\alpha)}\dddd{d^2}{dx^2}\int_b^x \dfrac{y(\xi)}{(x-\xi)^\alpha}d\xi=\dddd{d}{dx}D_{RL}^{\alpha} y(x)=
\dddd{d}{dx}y^{(\aaaa)}(x).$$
We obtain the following relation  for the shifted Gr\"unwald formulas and the Caputo derivative of the function $y(x)$.
$$h^{-\alpha}\Delta_{h,p}^\alpha y(x)= y^{(\alpha )}(x)+\left( p-\dfrac{\alpha}{2}\right)h \dddd{d}{dx}y^{(\aaaa)}(x)+O\left(h^2\right).  $$
From the mean value theorem for the function $y^{(\alpha )}(x)$ we have that
$$y^{(\alpha )}\left(x+ \left( p-\dfrac{\alpha}{2}\right)h\right)=y^{(\alpha )}(x)+\left( p-\dfrac{\alpha}{2}\right)h \dfrac{d}{dx}y^{(\alpha )}(x)+O\left(h^2\right).$$
Therefore
$$y^{(\alpha )}\left(x+ \left( p-\dfrac{\alpha}{2}\right)h\right)=h^{-\alpha}\Delta_{h,p}^\alpha y(x)+O\left(h^2\right).$$
Now we prove (ii). From \eqref{estimate} and $n=3$ we obtain
\begin{equation*}
\begin{aligned}
h^{-\alpha}\Delta_{h,p}^\alpha y(x)= D_{RL}^\aaaa y(x)+&\left( p-\dfrac{\alpha}{2}\right)h D_{RL}^{\alpha+1} y(x)+\\
&\llll(\dddd{\aaaa}{24}+\dddd{1}{2}\llll(p-\dddd{\aaaa}{2}\rrrr)^2\rrrr)h^2 D_{RL}^{2+\alpha} y(x)+O\left(h^3\right). 
\end{aligned}
\end{equation*}
We have that
$$D_{RL}^{2+\alpha} y(x)=\dfrac{1}{\Gamma (1-\alpha)}\dddd{d^3}{dx^3}\int_b^x \dfrac{y(\xi)}{(x-\xi)^\alpha}d\xi=\dddd{d^2}{dx^2}D_{RL}^{\alpha} y(x)=
\dddd{d^2}{dx^2}y^{(\aaaa)}(x).$$
Then
\begin{align*}
h^{-\alpha}\Delta_{h,p}^\alpha y(x)=&y^{(\aaaa)}(x)+\left( p-\dfrac{\alpha}{2}\right)h \dddd{d}{dx}y^{(\aaaa)}(x)+\\
&\dddd{1}{2}\llll(p-\dddd{\aaaa}{2}\rrrr)^2h^2 \dddd{d^2}{d x^2}y^{(\aaaa)}(x)+\dddd{\aaaa}{24}h^2\dddd{d^2}{d x^2}y^{(\aaaa)}(x)+O\left(h^3\right). 
\end{align*}
From the Mean-Value Theorem for $y^{(\aaaa)}(x)$ we obtain
$$h^{-\alpha}\Delta_{h,p}^\alpha y(x)=y^{(\alpha )}\left(x+ \left( p-\dfrac{\alpha}{2}\right)h\right)+\dddd{\aaaa}{24}h^2y^{(2+\alpha )}+O\left(h^3\right).$$
\end{proof}
 The most useful special case of \eqref{t1_1} is when $p=0$.
\begin{cor} Let $y(b)=y'(b)=y''(b)=0$ and $y\in C^{4}[b,x]$. Then
\begin{equation} \label{c2_1}
h^{-\alpha}\Delta_{h}^\alpha y(x)=y^{(\alpha )} \left( x-\dfrac{\alpha h}{2} \right)+O\left(h^2\right).
\end{equation}
\end{cor}
When $\aaaa=1$ and $\aaaa=2$ approximation \eqref{c2_1} becomes a central difference approximation for the first and second derivatives of the function $y(x)$
$$y'\llll(x-\dddd{h}{2} \rrrr)=\dddd{y(x)-y(x-h)}{h}+O\left(h^2\right),$$
$$y''\llll(x-h \rrrr)=\dddd{y(x)-2y(x-h)+y(x-2h)}{h^2}+O\left(h^2\right).$$
When $\aaaa=n$, where $n$ is a positive integer, the weights $w^{(\aaaa)}_k=0$ for $k>n$. This case is discussed in \cite{Gorenflo1996}.
Experimental results suggest that when the conditions of Theorem 1 are not satisfied,  the order of approximation \eqref{t1_1}  fluctuates.  While in some cases the order may still be two, in many cases it is lower, even the order may be one and lower than one (Section 4.4.2). Let $y(x)$ be a sufficiently differentiable function. 
\begin{clm} (Approximations for values of a function on a uniform grid) \label{average}
\begin{equation} \label{average2}
y_{n-\bbbb}=\bbbb y_{n-1}+(1-\bbbb)y_n+O\llll( h^2\rrrr),
\end{equation} 
\begin{equation} \label{average3}
y_{n-\beta}=\dddd{1}{2}\beta(\beta-1)y_{n-2}+\bbbb(2-\beta)y_{n-1}+\dddd{1}{2}(\beta-1)(\beta-2)y_{n}+O\llll( h^3\rrrr).
\end{equation}
\end{clm}
\begin{proof} From the Mean Value Theorem there exist numbers $\theta_1$ and $\theta_2$, such that
$$y_n=y_{n-\bbbb}+\bbbb h y'_{n-\bbbb}+\dddd{\bbbb^2 h^2}{2}  y''_{n-\theta_1},$$
$$y_{n-1}=y_{n-\bbbb}-(1-\bbbb) h y'_{n-\bbbb}+\dddd{(1-\bbbb)^2 h^2}{2}  y''_{n-\theta_2}.$$
Hence,
$$\bbbb y_{n-1}+(1-\bbbb)y_n=y_{n-\beta}+\dddd{(1-\bbbb)\bbbb^2 h^2}{2}  y''_{n-\theta_1}+\dddd{\bbbb(1-\bbbb)^2 h^2}{2}  y''_{n-\theta_2}.$$
Let $D_2$ be an upper bound for the second derivative. Then
$$\llll|\bbbb y_{n-1}+(1-\bbbb)y_n-y_{n-\beta}\rrrr|\leq \dddd{(1-\bbbb)\bbbb D_2 h^2}{2}.$$
 The proof for approximation (18) uses third order expansions and is similar to the proof of (17).
\end{proof}
\begin{lem} (Approximations for the shifted Gr\"unwald formulas)

 (i) Let $y(b)=y'(b)=y''(b)=0$ and $y\in C^{4}[b,x]$. Then
\begin{equation} \label{app2}
h^{-\alpha}\Delta_{h,p}^\alpha y_n=
\llll(\dfrac{\aaaa}{2}-p  \rrrr)y^{(\alpha )}_{n-1}+
\llll(1+p-\dfrac{\aaaa}{2} \rrrr)y^{(\alpha )}_{n}+O\left(h^2\right);
\end{equation}

 (ii) Let $y(b)=y'(b)=y''(b)=y'''(b)=0$ and $y\in C^{5}[b,x]$. Then
\begin{equation} \label{app3}
h^{-\alpha}\Delta_{h,p}^\alpha y_n=
\oooo{\beta}_1 y_{n-2}^{(\aaaa)}+\oooo{\beta}_2 y_{n-1}^{(\aaaa)}+\oooo{\bbbb}_3 y_n^{(\aaaa)}+O\left(h^3\right),
\end{equation}
where $\oooo{\bbbb}_1=\frac{p}{2}+\frac{p^2}{2}-\frac{5 \alpha }{24}-\frac{p \alpha }{2}+\frac{\alpha ^2}{8}$ and
$$\oooo{\bbbb}_2=-2 p-p^2+\frac{11 \alpha }{12}+p \alpha -\frac{\alpha ^2}{4},\quad
\oooo{\bbbb}_3=1+\frac{3 p}{2}+\frac{p^2}{2}-\frac{17 \alpha }{24}-\frac{p \alpha }{2}+\frac{\alpha ^2}{8}.$$
\end{lem}
\begin{proof} From Theorem 1(i) with $x=x_n$ we have
$$h^{-\alpha}\Delta_{h,p}^\alpha y_n=y^{(\alpha )} \left(x_n+\left(p-\dfrac{\alpha}{2}\right) h\right)+O\left(h^2\right).$$
From  \eqref{average2} with $\bbbb=\frac{\aaaa}{2}-p$ we obtain
$$h^{-\alpha}\Delta_{h,p}^\alpha y_n=
\llll(\dfrac{\aaaa}{2}-p  \rrrr)y^{(\alpha )}_{n-1}+
\llll(1+p-\dfrac{\aaaa}{2} \rrrr)y^{(\alpha )}_{n}+O\left(h^2\right).$$
Now we use the formula from Theorem 1(ii) to determine a third order approximation for the shifted Gr\"unwald formulas.
$$h^{-\alpha}\Delta_{h,p}^\alpha y_n=y^{(\alpha )} \left(x_n+\left(p-\dfrac{\alpha}{2}\right) h\right)+\dddd{\aaaa}{24}h^2y^{(2+\aaaa)}_n+O\left(h^3\right).$$
The central difference approximation for $y^{(2+\aaaa)}_n$ with nodes $\{x_{n-2},x_{n-1},x_n\}$ has order $O(h)$
$$y^{(2+\aaaa)}_n=\dddd{y^{(\aaaa)}_{n}-2y^{(\aaaa)}_{n-1}+y^{(\aaaa)}_{n-2}}{h^2}+O(h),$$
Then
$$h^{-\alpha}\Delta_{h,p}^\alpha y_n=y^{(\alpha )} \left(x_n+\left(p-\dfrac{\alpha}{2}\right) h\right)+\dddd{\aaaa h^2}{24}\llll(\dddd{y^{(\aaaa)}_{n}-2y^{(\aaaa)}_{n-1}+y^{(\aaaa)}_{n-2}}{h^2}+O\left(h\right)\rrrr)+O\left(h^3\right).$$
From  \eqref{average3} with $\bbbb=\frac{\aaaa}{2}-p$ we obtain approximation \eqref{app3}.
\end{proof}
 We obtain  second and third order approximations for the Gr\"unwald formula  from \eqref{app2} and $\eqref{app3}$ with $p=0$.
\begin{cor} \label{c5} (Approximations for the Gr\"unwald formula)

 (i) Let $y(b)=y'(b)=y''(b)=0$ and $y\in C^{4}[b,x]$. Then
$$h^{-\alpha}\Delta_{h}^\alpha y_n=
\llll(\dfrac{\aaaa}{2}  \rrrr)y^{(\alpha )}_{n-1}+
\llll(1-\dfrac{\aaaa}{2} \rrrr)y^{(\alpha )}_{n}+O\left(h^2\right);$$
(ii) Let $y(b)=y'(b)=y''(b)=y'''(b)=0$ and $y\in C^{5}[b,x]$. Then
$$h^{-\aaaa}\Delta_h^\aaaa y_n=\llll(\dddd{a^2}{8}-\dddd{5a}{24}\rrrr)y^{(\aaaa)}_{n-2}+\llll(\dddd{11a}{12}-\dddd{a^2}{4}\rrrr)y^{(\aaaa)}_{n-1}+\llll(1-\dddd{17a}{24}+\dddd{a^2}{8}\rrrr)y^{(\aaaa)}_{n}+O\llll( h^3\rrrr).$$
\end{cor}
In Corollary 2 and Corollary 5 we determined second and third order approximations for the Gr\"unwald formula using  values of Caputo derivatives. The three approximations  are suitable for  algorithms for numerical  solution of fractional differential equations on  a uniform grid. In the next lemma, we discuss a property of the Caputo derivative of a continuously differentiable function in a neighborhood of the lower limit $b$.
\begin{lem} \label{l6} Let $0<\aaaa<1$ and $y\in C^1[b,b+\epsilon]$, where $\epsilon>0$. Then
$$y^{(\aaaa)}(b)=0.$$
\end{lem}
\begin{proof} The function $y'$ is  bounded on the interval 
$[b,b+\epsilon]$. Let
$$\delta=\max_{b\leq x\leq b+\epsilon}y'(x).$$
 When $b<x<b+\epsilon$ we have 
$$\llll|y^{(\aaaa)}(x)\rrrr|\leq \dddd{1}{\Gamma(1-\aaaa)}\int_{b}^{x}\dddd{\llll|y'(\xi)\rrrr|}{(x-\xi)^\aaaa}d\xi\leq
\dddd{\delta}{\Gamma(1-\aaaa)}\int_b^x (x-\xi)^{-\aaaa}d\xi,$$
\begin{equation}\label{eq_l6}
\llll|y^{(\aaaa)}(x)\rrrr|\leq 
\dddd{\delta}{\Gamma(1-\aaaa)}\llll.-\dddd{(x-\xi)^{1-\aaaa}}{1-\aaaa} \rrrr|_b^x\leq 
\dddd{\delta (x-b)^{1-\aaaa}}{\Gamma(2-\aaaa)}.
\end{equation}
From the squeeze low of limits
$$0\leq \lim_{x\downarrow b} \llll|y^{(\aaaa)}(x)\rrrr|\leq 
\lim_{x\downarrow b} \dddd{\delta (x-b)^{1-\aaaa}}{\Gamma(2-\aaaa)}=0.$$
Hence,
$$y^{(\aaaa)}(b)=\lim_{x\downarrow b}y^{(\aaaa)}(x)=0.$$
\end{proof}
In the next two sections we compute approximations for the solutions of equations (1) and (2) when $b=0$ and the solutions are sufficiently differentiable functions. We use the result from Lemma 6 to determine the values of the  derivatives of the solution of equation (2) at the lower limit $x=0$, and the partial derivatives $u_t(x,0)$ and $u_{tt}(x,0)$ of the solution of the fractional sub-diffusion equation (1) when $t=0$.
\section{Numerical solution of  ordinary fractional differential equations}
The field of numerical computations for fractional differential equations has been rapidly gaining popularity for the last several decades. The fractional differential equations have diverse algorithms for computation of numerical solutions and a potential for practical applications. 
The analytical solutions of linear fractional differential equations with constant coefficients are determined with the integral transforms method \cite{Podlubny1999}. The algorithms for numerical solution can be used for a much larger class of equations including fractional differential equations with non-constant coefficients. 
Dithelm et al. \cite{DiethelmFordFreed2002} proposed  a prediction-correction algorithm for numerical approximation for ordinary fractional differential equations with accuracy $O\llll(h^{\min\{2,1+\aaaa\} }\rrrr)$. Deng \cite{Deng2007} presented an improved  prediction-correction algorithm with accuracy $O\llll(h^{\min\{2,1+2\aaaa\} }\rrrr)$. Higher order prediction-correction algorithms are discussed in \cite{YanPalFord2013}.
While the number of computations for  numerical solution  of ordinary fractional differential equations is much smaller than the number of computations for partial fractional differential equations it is greater than the number of computations for ordinary differential equations. An acceptable approximation \eqref{Recurrence2} for the solution of equation (2) is  obtained from the Gr\"unwald formula approximation for the Caputo derivative. It converges to the exact solution with accuracy $O(h)$. An improved approximation \eqref{Reccurence1} with accuracy $O\llll(h^{2-\aaaa}\rrrr)$ is obtained when we use approximation \eqref{CaputoApproximation}  instead of the  Gr\"unwald formula. Numerical experiments  for equation \eqref{eq22} and  approximations \eqref{Recurrence2} and \eqref{Reccurence1} on the interval $[0,1]$ are given in Table 1. In Figure 1 the two approximations are compared with the exact solution and the second order approximation \eqref{Recurrence3},  when $h=0.1$. 

In the present section we use approximations \eqref{SecondOrder2} and \eqref{ThirdOrder} to obtain recurrence relations for second and third order approximations to the solution of ordinary fractional differential equation (2) on the interval $[0,1]$,
	\begin{equation*} 
	y^{(\alpha )}(x)+y(x)=f(x)
	\end{equation*}
	 and we give a proof for the convergence of the algorithm. We can assume that equation \eqref{eq2} has initial condition
	$$y(0)=0$$
	because the function $\oooo{y}(x)=y(x)-y(0)$ is a solution of equation \eqref{eq2} with right-hand side $\oooo{f}(x)=f(x)-y(0)$. The exact solution of  equation \eqref{eq2} is obtained with the Laplace  transform method 
	\cite{Podlubny1999}
	$$y(x)=\int_{0}^{x}\xi^{\aaaa-1} E_{\aaaa,\aaaa}\llll(-\xi^\aaaa\rrrr)f(x-\xi)d\xi.$$
	An alternative approach to nonzero initial condition is to use approximations \eqref{SecondOrder},\eqref{SecondOrder2} and \eqref{ThirdOrder} for the function $y(x)-y(0)$ (as in formulas \eqref{ShiftedGrunwaldApproximation} and \eqref{2ndOrdApp}).
	
	 Let $h=T/N$ where $T>0$ and $N$ is a positive integer. By approximating the fractional derivative at the point $x_n=n h$ using the Gr\"unwald formula we obtain
$$\dfrac{1}{h^\alpha} \sum_{k=0}^n w_k^{(\alpha)} y_{n-k}+y_n\approx f_n,$$
$$ y_n+ \sum_{k=1}^n w_k^{(\alpha)} y_{n-k}+h^\alpha y_n\approx h^\alpha f_n.$$
The truncation errors of the two approximations are $O(h)$ and $O\llll(h^{1+\aaaa}\rrrr)$.
We compute an approximation $\widetilde{y}_n$ to the exact solution of \eqref{eq2} at the point $x_n$  with $\widetilde{y}_0=0$ and the recurrence relations  
	\begin{equation}\label{Recurrence2}
\widetilde{y}_n=\dfrac{1}{1+h^\alpha} \left( h^\alpha f_n-\sum_{k=1}^n 
w_k^{(\alpha)} \widetilde{y}_{n-k} \right).
\end{equation}
Approximation \eqref{Recurrence2} has accuracy  $O(h)$. 
	By approximating at the Caputo derivative at the points $x_n$   with \eqref{CaputoApproximation} we obtain similar recurrence relations  
\begin{equation}\label{Reccurence1}
\widetilde{y}_n=\dfrac{1}{h^\alpha+c_0^{(\alpha)}} \left( h^\alpha f_n-\sum_{k=1}^n c_k^{(\alpha)} \widetilde{y}_{n-k} \right).
\end{equation}
The accuracy of approximation \eqref{Reccurence1} is $O\llll(h^{2-\aaaa}\rrrr)$.
We evaluate numerically the performance of approximations \eqref{Recurrence2} and \eqref{Reccurence1} for the  equation
	\begin{equation}\label{eq22}
	y^{(\aaaa)}(x)+y(x)=2x^{2+\aaaa}+\Gamma(3+\aaaa)x^2.  
	\end{equation}
	Equation \eqref{eq22} has solution $y(x)=2x^{2+\aaaa}$.
	When $h=0.1$ and $\alpha=2/3$ the maximum error of approximations \eqref{Recurrence2}  and \eqref{Reccurence1} are  $0.111521$ and $0.0545347$.  In Table 1 we compute the maximum  errors and the order of approximations  \eqref{Recurrence2}  and \eqref{Reccurence1} for $\aaaa=2/3$ and different values of $h$. The graphs of  approximations \eqref{Recurrence2}  and \eqref{Reccurence1} (filled squares and empty circles) and the  solution of \eqref{eq22} on the interval $[0,1]$ are given in Figure 1.
\begin{table}
    \caption{Maximum error and order of  approximations \eqref{Recurrence2} and \eqref{Reccurence1} for equation \eqref{eq22} on the interval $[0,1]$ when  $\alpha=2/3$.}
    \begin{subtable}{0.5\linewidth}
      \centering
  \begin{tabular}{l c c }
  \hline \hline
    $h$ & $Error$ & $Order$  \\ 
		\hline \hline
$0.05$     &  $0.0560953$ & $0.991369$ \\
$0.025$    &  $0.0281318$ & $0.995677$ \\
$0.0125$   &  $0.0140870$ & $0.997837$ \\
$0.00625$  &  $0.0070488$ & $0.998918$  \\
$0.003125$ &  $0.0035257$ & $0.999459$  \\
		\hline
  \end{tabular}
    \end{subtable}%
    \begin{subtable}{.5\linewidth}
      \centering
				\quad
  \begin{tabular}{ l  c  c }
    \hline \hline
    $h$ & $Error$ &$Order$  \\ \hline \hline
$0.05$     & $0.0223527$    & $1.28672$ \\
$0.025$    & $0.0090448$    & $1.30529$ \\
$0.0125$   & $0.0036319$   & $1.31636$ \\
$0.00625$  & $0.0014517$   & $1.32299$ \\
$0.003125$ & $0.0005786$   & $1.32699$ \\ \hline
  \end{tabular}
    \end{subtable} 
\end{table}
\subsection{Second-order approximations}
In Corollary 5 we determined a second order approximation \eqref{SecondOrder2} for the Gr\"unwald formula with average values of Caputo derivatives. We use the approximation and \eqref{average2} to obtain second order numerical solutions \eqref{Recurrence3} and \eqref{Recurrence4} for equation (2). 
$$h^{-\alpha}\Delta_{h}^\alpha y_n=
\llll(\dfrac{\aaaa}{2}  \rrrr)y^{(\alpha )}_{n-1}+
\llll(1-\dfrac{\aaaa}{2} \rrrr)y^{(\alpha )}_{n}+O\llll( h^2\rrrr).$$
From equation (2) we have
$$y^{(\alpha )}_{n-1}=f_{n-1}-y_{n-1},\quad y^{(\alpha )}_{n}=f_{n}-y_{n}.$$
Then
$$\Delta_{h}^\alpha y_n=
h^{\alpha}\llll(\dfrac{\aaaa}{2}  \rrrr)(f_{n-1}-y_{n-1})+h^{\alpha}
\llll(1-\dfrac{\aaaa}{2} \rrrr)(f_{n}-y_{n})+O\llll( h^{2+\aaaa}\rrrr).$$
The Gr\"unwald formula for $y_n$ is defined as
$$\Delta_h^\aaaa y_n=\sum_{k=0}^{n} w_k^{(\aaaa)} y_{n-k}=y_n+\sum_{k=1}^{n} 
w_k^{(\aaaa)} y_{n-k}.$$
Let
$$\gamma=\dddd{1}{1+h^\aaaa\llll(1-\dddd{\aaaa}{2}\rrrr)}.$$
The exact solution of equation (2) satisfies
\begin{equation*}
y_n=\gamma
\llll(h^{\alpha}\llll(\dfrac{\aaaa}{2}  \rrrr)(f_{n-1}-y_{n-1})+h^{\alpha}
\llll(1-\dfrac{\aaaa}{2} \rrrr)f_{n} -\sum_{k=1}^{n} w_k^{(\aaaa)} y_{n-k}\rrrr)+O\llll( h^{2+\aaaa}\rrrr).
\end{equation*}
We can obtain a more convenient form of the above formula  using \eqref{average2}
$$f_{n-\frac{\aaaa}{2}}=\llll(\dddd{\aaaa}{2}\rrrr)f_{n-1}+\llll(1-\dddd{\aaaa}{2} \rrrr)f_n+O\llll( h^{2}\rrrr),$$
\begin{equation*}
y_n=\dfrac{1}{1+\left(1-\dfrac{\alpha}{2} \right)h^\alpha} \left( h^\alpha f_{n-\frac{\alpha}{2}}+\dddd{\aaaa}{2}\left( 2-h^\alpha\right) y_{n-1}-\sum_{k=2}^n w_k^{(\aaaa)} y_{n-k} \right)+O\llll( h^{2+\aaaa}\rrrr).
\end{equation*}

We compute  second order  approximations $\widetilde{y}_n$ to the exact solution   $y_n$ of equation (2) with $\wwww{y}_0=0$ and the  recurrence relations 
\begin{equation}\label{Recurrence3}
\widetilde{y}_n=\dfrac{1}{1+\left(1-\dfrac{\alpha}{2} \right)h^\alpha} \left( h^\alpha f_{n-\frac{\alpha}{2}}+\dddd{\aaaa}{2}\left( 2-h^\alpha\right)\widetilde{y}_{n-1}-\sum_{k=2}^n w_k^{(\aaaa)} \widetilde{y}_{n-k} \right),
\end{equation}
\begin{equation}\label{Recurrence4}
\widetilde{y}_n=\gamma \left( h^\alpha \llll(\llll(1-\dddd{\aaaa}{2} \rrrr)f_{n}+ \dddd{\aaaa}{2} f_{n-1}\rrrr)+\dddd{\aaaa}{2}\left( 2-h^\alpha\right)\widetilde{y}_{n-1}-\sum_{k=2}^n w_k^{(\aaaa)} \widetilde{y}_{n-k} \right).
\end{equation}

\begin{thm}\label{t7} Suppose that the solution $y(x)$ of equation (2) is sufficiently differentiable function on an interval $[0,T]$ and
$$y(0)=y'(0)=y''(0)=0.$$
 Then approximations \eqref{Recurrence3} and  \eqref{Recurrence4} converge to the solution of  equation \eqref{eq2} with accuracy $O\llll(h^2\rrrr)$.
\end{thm}
The proof of Theorem \ref{t7} is similar to the proof of Theorem \ref{t14}. The two numerical solutions \eqref{Recurrence3} and  \eqref{Recurrence4} have second order accuracy $O\llll(h^2\rrrr)$. Experimental results suggest that in many cases the error of approximation \eqref{Recurrence3} is smaller than the error of approximation  \eqref{Recurrence4} because it has a smaller truncation error.
In Table 2 we compute the maximum error and order of    approximation \eqref{Recurrence3} for equation \eqref{eq22} and different values of $h$.
When $h=0.1$ and $\alpha=2/3$ the error of  approximation \eqref{Recurrence3} is $ 0.005828$.  In Figure 2 we compare approximations \eqref{Recurrence2},\eqref{Reccurence1}  and \eqref{Recurrence3} with the exact solution of equation \eqref{eq22}.
\begin{figure}[ht]
  \centering
  \caption{Graph of the exact solution of equation \eqref{eq22} and  approximations \eqref{Recurrence2}, \eqref{Reccurence1} and \eqref{Recurrence3} for $h=0.1$ and $\alpha=2/3$.} 
  \includegraphics[width=0.55\textwidth]{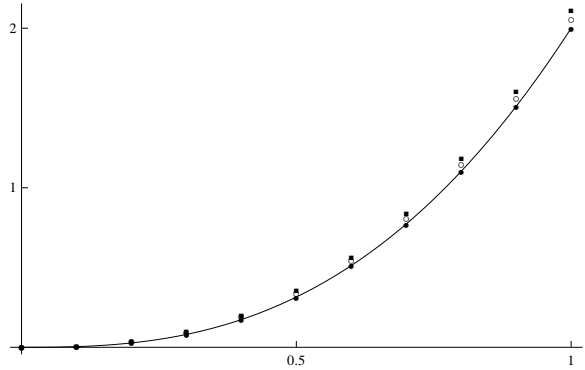}
\end{figure}
	\begin{table}
	\caption{Maximum error and order of approximation \eqref{Recurrence3} for equation \eqref{eq22} on the interval $[0,1]$ and  $\alpha=2/3$.}
	\centering
  \begin{tabular}{ l  c   c c }
    \hline \hline
    $h$  & $Error$                 & $Ratio$ & $\log_2 (Ratio)$  
		\\ \hline \hline
$0.05$     & $0.00146501$            & $3.97836$  & $1.99217$   \\ 
$0.025$    & $0.00036724$           & $3.98921$  & $1.99610$   \\ 
$0.0125$   & $0.00009193$           & $3.99461$  & $1.99805$   \\ 
$0.00625$  & $0.00002210$          & $3.99731$  & $1.99903$   \\ 
$0.003125$ & $5.75\times 10^{-6}$  & $3.99865$  & $1.99951$   \\ \hline
  \end{tabular}
	\end{table}
	
Suppose that the solution $y(x)$ of  equation (2) is sufficiently differentiable function. Denote
$$L_1=y'(0),\quad L_2=y''(0),\quad L_3=y'''(0).$$
In the next lemma we determine  $L_1$ and $L_2$ from the function $f(x)$.
\begin{lem} Suppose that $y(x)$ is a sufficiently differentiable solution of   \eqref{eq2}. 
		  $$f(0)=0,\quad L_1=\lim_{x \downarrow 0} f^{(1-\aaaa)}(x),$$
	 $$L_2=\lim_{x\downarrow 0}\llll( \dddd{d}{dx}f^{(1-\aaaa)}(x)-
	\dfrac{L_1 x^{\aaaa-1}}{\Gamma (\aaaa)}\rrrr).$$
\end{lem}
\begin{proof} From Lemma \ref{l6} we have that
$$y^{(\aaaa)}(0)=0,\quad\text{and}\quad y^{(1-\aaaa)}(0)=0.
$$
Then 
$$0=y(0)=f(0)-y^{(\aaaa)}(0)=f(0).$$
By applying fractional differentiation of order $1-\aaaa$ to both sides of \eqref{eq2} 
$$y'(x)+y^{(1-\aaaa)}(x)=f^{(1-\aaaa)}(x).$$
Hence,
$$L_1=y'(0)=f^{(1-\aaaa)}(0)-y^{(1-\aaaa)}(0)=f^{(1-\aaaa)}(0)=\lim_{x \downarrow 0} f^{(1-\aaaa)}(x).$$
 By differentiating equation \eqref{eq2} we obtain
\begin{align}\label{fdeq}
y''(x)+y^{(1+(1-\aaaa))}(x)=f^{(1+(1-\aaaa))}(x).
\end{align}
We determine the value of the Miller-Ross derivative $y^{(1+(1-\aaaa))}(0)$ using differentiation and integration by parts.
$$y^{(1+(1-\aaaa))}(x)=\dddd{d}{dx}y^{(1-\aaaa)}(x)=\dddd{1}{\Gamma(\aaaa)}\dddd{d}{dx}\llll(\int_{0}^x \dddd{y'(\xi)}{(x-\xi)^{1-\aaaa}}d\xi\rrrr),$$
$$\Gamma(\aaaa)y^{(1+(1-\aaaa))}(x)=-\dddd{d}{dx}\llll(\int_{0}^x y'(\xi)d\frac{(x-\xi)^{\aaaa}}{\aaaa}\rrrr),$$
$$\Gamma(1+\aaaa)y^{(1+(1-\aaaa))}(x)=-
\dddd{d}{dx}\llll( \llll. y'(\xi)(x-\xi)^{\aaaa}\rrrr]_0^x
-\int_{0}^x y''(\xi) (x-\xi)^{\aaaa}d\xi\rrrr),$$
$$\Gamma(1+\aaaa)y^{(1+(1-\aaaa))}(x)=-
\dddd{d}{dx}\llll( - y'(0) x^{\aaaa}
-\int_{0}^x y''(\xi) (x-\xi)^{\aaaa}d\xi\rrrr),$$
\begin{align}\label{form27}
y^{(1+(1-\aaaa))}(x)=\dddd{y'(0) x^{\aaaa-1}}{\Gamma(\aaaa)}+
\dddd{1}{\Gamma(\aaaa)}\int_{0}^x \dddd{y''(\xi)}{(x-\xi)^{1-\aaaa}}d\xi,
\end{align}
$$y^{(1+(1-\aaaa))}(x)=\dddd{y'(0) x^{\aaaa-1}}{\Gamma(\aaaa)}+
D_x^{2-\aaaa}y(x).$$
From \eqref{fdeq},
$$y''(x)=f^{(1+(1-\aaaa))}(x)-\dddd{L_1 x^{\aaaa-1}}{\Gamma(\aaaa)}-D_x^{2-\aaaa}y(x).$$
We have that $D_x^{2-\aaaa}y(0)=0$, when $y\in C^2[0,1]$. Hence,
$$L_2=y''(0)=\lim_{x\rightarrow 0}\llll(f^{(1+(1-\aaaa))}(x)-\dddd{L_1 x^{\aaaa-1}}{\Gamma(\aaaa)}\rrrr).$$
\end{proof}
Let 
\begin{align}\label{e27}
z(x)=y(x)-L_1x-\dddd{L_2}{2} x^2.
\end{align}
\begin{lem} The function $z(x)$ satisfies $z(0)=z'(0)=z''(0)=0$ and is a solution of the ordinary  fractional differential equation
\begin{align}\label{eqn222}
z^{(\aaaa)}(x)+z(x)=F(x),
\end{align}
where
$$F(x)=f(x)-L_1 x-\dddd{L_2}{2} x^2-\dfrac{L_1}{\Gamma (2-\aaaa)}x^{1-\aaaa}-\dfrac{ L_2}{\Gamma (3-\aaaa)}x^{2-\aaaa}.$$
\end{lem}
\begin{proof} $z(0)=y(0)=0$. By differentiating \eqref{e27}
$$z'(0)=y'(0)-L_1=0,\quad z''(0)=y''(0)-L_2=0.$$
The function $z(x)$ has fractional derivative of order $\aaaa$
$$z^{(\aaaa)}(x)=y^\aaaa(x)-\dfrac{L_1}{\Gamma (2-\aaaa)}x^{1-\aaaa}-
\dfrac{ L_2}{\Gamma (3-\aaaa)}x^{2-\aaaa}.$$
Then
$$z^{(\aaaa)}(x)+z(x)=y^{(\aaaa)}(x)+y(x)-L_1 x-\dddd{L_2}{2} x^2-\dfrac{L_1}{\Gamma (2-\aaaa)}x^{1-\aaaa}-\dfrac{ L_2}{\Gamma (3-\aaaa)}x^{2-\aaaa},$$
$$z^{(\aaaa)}(x)+z(x)=f(x)-L_1 x-\dddd{L_2}{2} x^2-\dfrac{L_1}{\Gamma (2-\aaaa)}x^{1-\aaaa}-\dfrac{L_2}{\Gamma (3-\aaaa)}x^{2-\aaaa}.$$
\end{proof}
Approximations \eqref{Recurrence3} and \eqref{Recurrence4} are second order numerical solutions of equation \eqref{eqn222}, because the function $z(x)$ satisfies the conditions of Lemma 7.
\subsection{Estimates for Gr\"unwald weights}
In section 4.1 we determined recurrence relations 
\eqref{Recurrence3}  and \eqref{Recurrence4} for second order numerical solutions of equation (2). In section 4.3 and section 5 we obtain a third order numerical solution of equation (2) and second order difference approximations for the fractional sub-diffusion equation.  In Theorem \ref{t7}, Theorem \ref{t14}  and Theorem \ref{t29} we discuss the convergence properties of the approximations. The proofs rely on the lower bound \eqref{Estimate} for the tail of the sum of Grunwald weights \eqref{series}.

The Gr\"unwald weights $w_n^{(\aaaa)}$ are computed recursively as
$$w_0^{(\aaaa)}=1,\quad w_1^{(\aaaa)}=-\aaaa,\quad w_2^{(\aaaa)}=\dfrac{\aaaa(\aaaa-1)}{2},$$
$$w_n^{(\aaaa)} =(-1)^n \binom{\aaaa}{n} =\llll( 1-\dfrac{\aaaa+1}{n}\rrrr) w_{n-1}^{(\aaaa)},$$
where $\llll\{w_n^{(\aaaa)}\rrrr\}_{n=1}^\infty$ is an increasing sequence of negative numbers with sum 
\begin{equation}\label{series}
\sum_{n=1}^\infty w_n^{(\aaaa)}=-1.
\end{equation}
The Gr\"unwald weights converge to zero with an asymptotic rate \cite{MeerschaertSikorskii2011} 
$$\llll|w_n^{(\aaaa)}\rrrr|\sim
 \dddd{\aaaa}{\Gamma(1-\aaaa)}\dddd{1}{n^{1+\aaaa}} \quad \text{ as }n\rightarrow \infty.$$
We determine bounds for the Gr\"unwald weights using the properties of the exponential function.
\begin{lem} \label{Inequalities} (Inequalities for the exponential function)

	(i) $1-x<e^{-x}$ for $0<x<1$;
	
	(ii) $1-x>e^{-x-x^2}$ for $0<x<2/3$.
\end{lem}
\begin{proof} Let $h(x)=e^{-x}+x-1$. The function $h(x)$ has a positive derivative.
$$h'(x)=1-e^{-x}> 1-e^0=0.$$
Therefore, 
$$h(x)=e^{-x}+x-1>h(0)=0.$$
Now we prove  (ii). By taking logarithm from both sides
$$\ln(1-x)>-x-x^2.$$
The function $\ln(1-x)$ has Maclaurin series 
$$\ln (1-x)=-\sum_{n=1}^\infty \dddd{x^n}{n}=-x-\dddd{x^2}{2}-\sum_{n=3}^\infty \dddd{x^n}{n}.$$
Inequality (ii) is equivalent to
$$\dddd{x^2}{2}>\sum_{n=3}^\infty \dddd{x^n}{n},$$
$$\dddd{1}{2}>\dddd{x}{3}+\dddd{x^2}{4}+\dddd{x^3}{5}+\sum_{n=6}^\infty \dddd{x^{n-2}}{n}.$$
We have that $0<x<2/3$. Then
\begin{equation*}
\begin{aligned}
\dddd{x}{3}+\dddd{x^2}{4}+\dddd{x^3}{5}&+\sum_{n=6}^\infty \dddd{x^{n-2}}{n}<\dddd{1}{3}\llll(\dddd{2}{3}\rrrr)+\dddd{1}{4}\llll(\dddd{2}{3} \rrrr)^2+\dddd{1}{5}\llll(\dddd{2}{3} \rrrr)^3+\sum_{n=6}^\infty \dddd{1}{n}\llll(\dddd{2}{3} \rrrr)^{n-2}\\
&<\dddd{53}{135}+\dddd{1}{6}\llll(\dddd{2}{3} \rrrr)^4\sum_{n=6}^\infty \dddd{6}{n}\llll(\dddd{2}{3} \rrrr)^{n-6}<\dfrac{53}{135}+\dfrac{32}{2187}\sum_{n=6}^\infty \llll(\dddd{2}{3} \rrrr)^{n-6},\\
\dddd{x}{3}+\dddd{x^2}{4}+\dddd{x^3}{5}&+\sum_{n=6}^\infty \dddd{x^{n-2}}{n}<\dfrac{53}{135}+\dfrac{96}{2187}=\dfrac{1591}{3645}<\dddd{1}{2}.
\end{aligned}
\end{equation*}
\end{proof}
In the next two lemmas we determine upper and lower bounds for $\llll|w_n^{(\aaaa)}\rrrr|$ and $\sum_{k=n}^\infty \llll|w_k^{(\aaaa)}\rrrr|$.
\begin{lem} (Estimates for Gr\"unwald weights)
$$e^{-(\aaaa+1)^2\llll( \frac{\pi^2}{6}-\frac{5}{4}\rrrr)}  \dfrac{ \aaaa (1-\aaaa) 2^{\aaaa}}{n^{\aaaa+1}}<\llll|w_n^{(\aaaa)}\rrrr| <\dfrac{\aaaa 2^{\aaaa+1}}{(n+1)^{\aaaa+1}}.$$
\end{lem}
\begin{proof} We determine  bounds for the Gr\"unwald weights from the recursive  formula  and the inequalities for the exponential function.
$$\llll|w_n^{(\aaaa)}\rrrr|=\llll( 1-\dfrac{\aaaa+1}{n}\rrrr) \llll|w_{n-1}^{(\aaaa)} \rrrr|< e^{-\frac{\aaaa+1}{n}}\llll| w_{n-1}^{(\aaaa)} \rrrr|,$$
$$\llll|w_n^{(\aaaa)}\rrrr|< e^{-\frac{\aaaa+1}{n}}\llll| w_{n-1}^{(\aaaa)} \rrrr|
<e^{-\frac{\aaaa+1}{n}}e^{-\frac{\aaaa+1}{n-1}}\llll| w_{n-2}^{(\aaaa)} \rrrr|, $$
$$\llll|w_n^{(\aaaa)}\rrrr|< e^{-\frac{\aaaa+1}{n}}e^{-\frac{\aaaa+1}{n-1}}\cdots e^{-\frac{\aaaa+1}{2}}\llll| w_{1}^{(\aaaa)} \rrrr|=\aaaa e^{-(\aaaa+1)\sum_{k=2}^{n}\frac{1}{k}}. $$
The function $1/x$ is  decreasing for $x\geq 0$. Then
$$\sum_{k=2}^{n}\dfrac{1}{k}>\int_2^{n+1}\dfrac{1}{x}dx=\ln (n+1)-\ln 2$$
and
$$\llll|w_n^{(\aaaa)}\rrrr|< \aaaa e^{-(\aaaa+1)\sum_{k=2}^{n}\frac{1}{k}} <
\aaaa e^{-(\aaaa+1)(\ln (n+1)-\ln 2)}=\dfrac{\aaaa 2^{\aaaa+1}}{(n+1)^{\aaaa+1}}.$$
Now we use the inequality for the exponential function from  Lemma \ref{Inequalities}(ii) to determine a lower bound for the Gr\"unwald weights.
$$\llll|w_n^{(\aaaa)}\rrrr|=\llll( 1-\frac{\aaaa+1}{n}\rrrr) \llll|w_{n-1}^{(\aaaa)} \rrrr|> e^{-\frac{\aaaa+1}{n}-\llll(\frac{\aaaa+1}{n}\rrrr)^2}\llll| w_{n-1}^{(\aaaa)} \rrrr|,$$
$$\llll|w_n^{(\aaaa)}\rrrr| > e^{-\frac{\aaaa+1}{n}-\llll(\frac{\aaaa+1}{n}\rrrr)^2}\llll| w_{n-1}^{(\aaaa)} \rrrr|
>e^{-\frac{\aaaa+1}{n}-\llll(\frac{\aaaa+1}{n}\rrrr)^2}e^{-\frac{\aaaa+1}{n-1}-\llll(\frac{\aaaa+1}{n-1}\rrrr)^2}\llll| w_{n-2}^{(\aaaa)} \rrrr|, $$
$$\llll|w_n^{(\aaaa)}\rrrr| > e^{-\frac{\aaaa+1}{n}-\llll(\frac{\aaaa+1}{n}\rrrr)^2}e^{-\frac{\aaaa+1}{n-1}-\llll(\frac{\aaaa+1}{n-1}\rrrr)^2}\cdots e^{-\frac{\aaaa+1}{3}-\llll(\frac{\aaaa+1}{3}\rrrr)^2}\llll| w_{2}^{(\aaaa)} \rrrr|, $$
$$\llll|w_n^{(\aaaa)}\rrrr| > \frac{\aaaa (1-\aaaa)}{2} e^{-(\aaaa+1)\sum_{k=3}^{n}\frac{1}{k}-(\aaaa+1)^2\sum_{k=3}^{n}\frac{1}{k^2}}, $$
$$\llll|w_n^{(\aaaa)}\rrrr| > \frac{\aaaa (1-\aaaa)}{2} e^{-(\aaaa+1)\sum_{k=3}^{n}\frac{1}{k}}e^{-(\aaaa+1)^2\sum_{k=3}^{\infty}\frac{1}{k^2}}. $$
We have that \cite{GradshteynRyzhik2007}
$$\sum_{k=3}^{\infty}\dfrac{1}{k^2}=\sum_{k=1}^{\infty}\dfrac{1}{k^2}-1-\dddd{1}{4}=\dddd{\pi^2}{6}-\dddd{5}{4}.$$
Then
$$\llll|w_n^{(\aaaa)}\rrrr| > \frac{\aaaa (1-\aaaa)}{2} e^{-(\aaaa+1)\sum_{k=3}^{n}\frac{1}{k}}e^{-(\aaaa+1)^2\llll( \frac{\pi^2}{6}-\frac{5}{4}\rrrr)}. $$
The function $1/x$ is  decreasing  for $x\geq 0$. Then
$$\sum_{k=3}^{n}\dfrac{1}{k}<\int_2^{n}\dfrac{1}{x}dx=\ln n-\ln 2.$$
Hence,
$$\llll|w_n^{(\aaaa)}\rrrr|> \dfrac{\aaaa (1-\aaaa)}{2}e^{-(\aaaa+1)^2\llll( \frac{\pi^2}{6}-\frac{5}{4}\rrrr)}   e^{-(\aaaa+1)(\ln n -\ln 2)},$$
$$\llll|w_n^{(\aaaa)}\rrrr|>e^{-(\aaaa+1)^2\llll( \frac{\pi^2}{6}-\frac{5}{4}\rrrr)}  \dfrac{ \aaaa (1-\aaaa) 2^{\aaaa}}{n^{\aaaa+1}}.
$$
\end{proof}
We use the estimates for the Gr\"unwald weights to determine bounds for the tail of \eqref{series}.
\begin{lem} \label{l12}(Bounds for sums of Gr\"unwald weights)
\begin{equation} \label{Estimate}
  \dfrac{1-\aaaa}{5} \llll(\dfrac{ 2}{ n}\rrrr)^\aaaa<\sum_{k=n}^\infty \llll|w_k^{(\aaaa)}\rrrr| <2\llll(\dfrac{ 2}{ n}\rrrr)^\aaaa.
\end{equation}
\end{lem}
\begin{proof} From Lemma 11 we have
 $$\sum_{k=n}^\infty e^{-(\aaaa+1)^2\llll( \frac{\pi^2}{6}-\frac{5}{4}\rrrr)}  \dfrac{ \aaaa (1-\aaaa) 2^{\aaaa}}{k^{\aaaa+1}}<\sum_{k=n}^\infty \llll|w_k^{(\aaaa)}\rrrr| <
\sum_{k=n}^\infty \dfrac{\aaaa 2^{\aaaa+1}}{(k+1)^{\aaaa+1}}.$$
The function $1/x^{\aaaa+1}$ is a decreasing for $x\geq 0$ and we have:
$$\sum_{k=n}^\infty \dfrac{1}{(k+1)^{\aaaa+1}}<\int_n^\infty \dfrac{1}{x^{\aaaa+1}}dx<\sum_{k=n}^\infty \dfrac{1}{k^{\aaaa+1}},$$
$$\sum_{k=n}^\infty \dfrac{1}{(k+1)^{\aaaa+1}}<\dfrac{1}{\aaaa n^\aaaa}<\sum_{k=n}^\infty \dfrac{1}{k^{\aaaa+1}}.$$
Hence,
\begin{equation*} 
e^{-(\aaaa+1)^2\llll( \frac{\pi^2}{6}-\frac{5}{4}\rrrr)}  \dfrac{  (1-\aaaa) 2^{\aaaa}}{n^{\aaaa}}<\sum_{k=n}^\infty \llll|w_k^{(\aaaa)}\rrrr| <\dfrac{ 2^{\aaaa+1}}{n^{\aaaa}}.
\end{equation*}
The number $\aaaa$ is between zero and one. Then
$$e^{-(\aaaa+1)^2\llll( \frac{\pi^2}{6}-\frac{5}{4}\rrrr)}>e^{-4\llll( \frac{\pi^2}{6}-\frac{5}{4}\rrrr)}>\dddd{1}{5}.$$
\end{proof}
\subsection{Third-order approximation}
In the present section we use approximation \eqref{ThirdOrder} to determine recurrence relations for a third order approximation for equation (2), when the solution is sufficiently  differentiable function. In section 4.2 we determined an estimate for sums of Gr\"unwald weights.  In Lemma \ref{ll13} and Theorem \ref{t14} we use the lower bound to prove that the approximation converges to the solution with  accuracy $O\llll( h^3 \rrrr)$, when the solution satisfies the conditions of Corollary \ref{c5}.
$$h^{-\aaaa}\Delta_h^\aaaa y_n=\llll(\dddd{a^2}{8}-\dddd{5a}{24}\rrrr)y^{(\aaaa)}_{n-2}+\llll(\dddd{11a}{12}-\dddd{a^2}{4}\rrrr)y^{(\aaaa)}_{n-1}+\llll(1-\dddd{17a}{4}+\dddd{a^2}{8}\rrrr)y^{(\aaaa)}_{n}+O\llll( h^3\rrrr),$$
\begin{equation*}
\begin{aligned}
\Delta_h^\aaaa y_n=h^{\aaaa}\llll(\dddd{a^2}{8}-\dddd{5a}{24}\rrrr)(f_{n-2}-y_{n-2})&+h^{\aaaa}\llll(\dddd{11a}{12}-\dddd{a^2}{4}\rrrr)(f_{n-1}-y_{n-1})+\\
&h^{\aaaa}\llll(1-\dddd{17a}{4}+\dddd{a^2}{8}\rrrr)(f_{n}-y_{n})+O\llll( h^{3+\aaaa}\rrrr).
\end{aligned}
\end{equation*}
We have that
$$\Delta_h^\aaaa y_n=\sum_{k=0}^{n} w_k^{(\aaaa)} y_{n-k}=y_n+\sum_{k=1}^{n} w_k^{(\aaaa)} y_{n-k}.$$
Let
$$\gamma=1+h^\aaaa\llll(1-\dddd{17\aaaa}{24}+\dddd{a^2}{8}\rrrr).$$
The solution of equation (2) satisfies
\begin{align*}
 y_n=\gamma^{-1}&
\llll[ h^{\aaaa}\llll(\dddd{11a}{12}-\dddd{a^2}{4}\rrrr)\rrrr.
(f_{n-1}-y_{n-1})+ h^{\aaaa}\llll(\dddd{a^2}{8}-\dddd{5a}{24}\rrrr)(f_{n-2}-y_{n-2}) \\
&+\llll. h^{\aaaa}\llll(1-\dddd{17a}{4}+\dddd{a^2}{8}\rrrr)f_{n}-\sum_{k=1}^{n} w_k^{(\aaaa)} y_{n-k}\rrrr]+O\llll( h^{3+\aaaa}\rrrr).
\end{align*}
We compute a numerical solution $\wwww{y}_n$ with $\wwww{y}_0=\wwww{y}_1=0$ and
\begin{align}
 \nonumber \wwww{y}_n=\gamma^{-1}\llll[ 
h^{\aaaa}\llll(\dddd{11a}{12}-\dddd{a^2}{4}\rrrr)\rrrr.&(f_{n-1}-\wwww{y}_{n-1})+ h^{\aaaa}\llll(\dddd{a^2}{8}-\dddd{5a}{24}\rrrr)(f_{n-2}-\wwww{y}_{n-2})\\
&+\llll. h^{\aaaa}\llll(1-\dddd{17a}{4}+\dddd{a^2}{8}\rrrr)f_{n}-\sum_{k=1}^{n} w_k^{(\aaaa)} \wwww{y}_{n-k}\rrrr]. \label{3rdOrder}
\end{align}
Let $e_n=y_n-\wwww {y}_n$ be the error of approximation $\eqref{3rdOrder}$. The numbers $e_n$ satisfy
\begin{align*}
 e_n=-\gamma^{-1}\llll[ h^{\aaaa}\llll(\dddd{11a}{12}-\dddd{a^2}{4}\rrrr)e_{n-1}+ h^{\aaaa}\llll(\dddd{a^2}{8}-\dddd{5a}{24}\rrrr)e_{n-2}+ \sum_{k=1}^{n} w_k^{(\aaaa)} e_{n-k}\rrrr] +A_n,
\end{align*}
where $A_n$ is the truncation error of approximation \eqref{3rdOrder}. The numbers $A_n$ satisfy
$$|A_n|<A h^{3+\aaaa},$$
where $A$ is a constant such that
$$A>\dddd{(1-\aaaa)2^{\aaaa}D_{3+\aaaa}}{10\Gamma{(4+\aaaa)}},$$
and 
$$D_{3+\aaaa}=\max_{0\leq x\leq1}\llll|y^{(3+\aaaa)}(x)\rrrr|.$$
We can represent the recurrence relations for the errors $e_n$ as
\begin{align*}\nonumber
 e_n=\gamma^{-1}\llll[ \llll( \aaaa-h^{\aaaa}\llll(\dddd{11a}{12}-\dddd{a^2}{4}\rrrr)\rrrr)\rrrr. e_{n-1}+&\llll(\dddd{\aaaa(1-\aaaa)}{2}+ h^{\aaaa}\llll(\dddd{5a}{24}-\dddd{a^2}{8}\rrrr)\rrrr)e_{n-2}\\
& -\llll.\sum_{k=3}^{n} w_k^{(\aaaa)} e_{n-k}\rrrr]+A_n, 
\end{align*}
\begin{align}
 e_n=\gamma^{-1}\llll(\gamma_1 e_{n-1}+\gamma_2 e_{n-2}+
\sum_{k=3}^{n} \gamma_k e_{n-k}\rrrr)+A_n, \label{error2}
\end{align}
where
$$\gamma_1= \aaaa-h^{\aaaa}\llll(\dddd{11a}{12}-\dddd{a^2}{4}\rrrr),
\gamma_2=\dddd{\aaaa(1-\aaaa)}{2}+ h^{\aaaa}\llll(\dddd{5a}{24}-\dddd{a^2}{8}\rrrr),
\gamma_k=-w_k^{(\aaaa)} (k\geq3).
$$
The numbers $\gamma_n$ are positive, because $h^\aaaa<1$ and $0<\aaaa<1$.
\begin{lem}\label{ll13}
	Suppose that equation \eqref{eq2} has sufficiently differentiable solution on the interval $[0,T]$ and $y(0)=y'(0)=y''(0)=y'''(0)=0$. Then
\begin{align}\label{l13}
|e_n|<\llll(\dddd{10A}{ (1-\aaaa)2^{\aaaa}}\rrrr)n^\aaaa h^{3+\aaaa}.
\end{align}
\end{lem}
	\begin{proof} We prove \eqref{l13} by induction on $n$. 
	$$|e_0|=|y_0-\wwww{y}_0|=0.$$
	From the Generalized Mean Value Theorem
	$$e_1=y_1-\wwww{y}_1=y(h)=\dddd{y^{(3+\aaaa)}(\xi)}{\Gamma(4+\aaaa)}h^{3+\aaaa},$$
	$$| e_1 |<\dddd{\llll|y^{(3+\aaaa)}(\xi)\rrrr|}{\Gamma(4+\aaaa)}h^{3+\aaaa}
	\leq \dddd{D_{3+\aaaa}}{\Gamma(4+\aaaa)}h^{3+\aaaa}<
	\dddd{10A}{(1-\aaaa)2^{\aaaa}}h^{3+\aaaa}.$$
	Suppose that \eqref{l13} holds for all $n \leq \oooo{n}-1$.
	\begin{align*}
 |e_{\oooo{n}}|\leq\gamma^{-1}\llll(\gamma_1 |e_{\oooo{n}-1}|+\gamma_2 |e_{\oooo{n}-2}|+
\sum_{k=3}^{\oooo{n}-1} \gamma_k |e_{\oooo{n}-k}|\rrrr) +|A_{\oooo{n}}|.
\end{align*}
By the induction hypothesis
$$|e_{\oooo{n}-k}|<\dddd{10A}{(1-\aaaa)2^{\aaaa}}(\oooo{n}-k)^\aaaa h^{3+\aaaa}<\dddd{10}{(1-\aaaa)2^{\aaaa}} \oooo{n}^\aaaa h^{3+\aaaa}\quad(k=1,\cdots,\oooo{n}-1).$$
Then
\begin{align}\label{L13_1}
|e_{\oooo{n}}|< \dddd{10A\gamma^{-1}}{(1-\aaaa)2^{\aaaa}}  \oooo{n}^\aaaa h^{3+\aaaa}
\sum_{k=1}^{\oooo{n}-1} \gamma_k  +A h^{3+\aaaa}.
\end{align}
We have that
$$\sum_{k=1}^{\oooo{n}-1} \gamma_k=-h^{\aaaa}\llll(\dddd{11a}{12}-\dddd{a^2}{4}\rrrr)+ h^{\aaaa}\llll(\dddd{5a}{24}-\dddd{a^2}{8}\rrrr)+\sum_{k=1}^{\oooo{n}-1} \llll|w_k^{(\aaaa)}\rrrr|,$$
$$\sum_{k=1}^{\oooo{n}-1} \gamma_k=h^{\aaaa}\llll(\dddd{a^2}{8}-\dddd{17a}{24}\rrrr)+\sum_{k=1}^{\infty} \llll|w_k^{(\aaaa)}\rrrr|-\sum_{k=\oooo{n}}^{\infty} \llll|w_k^{(\aaaa)}\rrrr|,$$
$$\sum_{k=1}^{\oooo{n}-1} \gamma_k=1+h^{\aaaa}\llll(\dddd{a^2}{8}-\dddd{17a}{24}\rrrr)-\sum_{k=\oooo{n}}^{\infty} \llll|w_k^{(\aaaa)}\rrrr|=\gamma-h^\aaaa-
\sum_{k=\oooo{n}}^{\infty} \llll|w_k^{(\aaaa)}\rrrr|.$$
Hence,
\begin{align}\label{L13_2}
\sum_{k=1}^{\oooo{n}-1} \gamma_k < \gamma-
\sum_{k=\oooo{n}}^{\infty} \llll|w_k^{(\aaaa)}\rrrr|.
\end{align}
From \eqref{L13_1} and \eqref{L13_2},
$$|e_{\oooo{n}}|< \dddd{10A \gamma^{-1}}{(1-\aaaa)2^{\aaaa}} \oooo{n}^\aaaa h^{3+\aaaa}\llll(
\gamma-\sum_{k=\oooo{n}}^{\infty} \llll|w_k^{(\aaaa)}\rrrr|
\rrrr)  +A h^{3+\aaaa}.$$
The number $\gamma=1+h^\aaaa\llll(1-\frac{17\aaaa}{24}+\frac{a^2}{8}\rrrr)$ satisfies
$$1<\gamma<2,$$
because $h^\aaaa<1$ and
$$1-\frac{17\aaaa}{24}+\frac{a^2}{8}>1-\frac{17}{24}>\frac{7}{24}>0,$$ 
$$1-\frac{17\aaaa}{24}+\frac{a^2}{8}<1-\frac{17\aaaa}{24}+\frac{17a^2}{24}<1.$$ 
Then 
$$\dddd{1}{2}<\gamma^{-1}<1$$
and
$$|e_{\oooo{n}}|\leq \dddd{10A}{(1-\aaaa)2^{\aaaa}} \oooo{n}^\aaaa h^{3+\aaaa}\llll(
1-\gamma^{-1}\sum_{k=\oooo{n}}^{\infty} \llll|w_k^{(\aaaa)}\rrrr|
\rrrr)  +A h^{3+\aaaa},$$
$$|e_{\oooo{n}}|< \dddd{10A}{(1-\aaaa)2^{\aaaa}} \oooo{n}^\aaaa h^{3+\aaaa}-\dddd{5A}{(1-\aaaa)2^{\aaaa}}  \oooo{n}^\aaaa h^{3+\aaaa}\sum_{k=\oooo{n}}^{\infty} \llll|w_k^{(\aaaa)}\rrrr|
 +A h^{3+\aaaa}.$$
In Lemma 12 we showed that
$$\sum_{k=\oooo{n}}^\infty \llll|w_k^{(\aaaa)}\rrrr| >\dddd{1-\aaaa}{5}\llll(\dfrac{ 2}{ \oooo{n}}\rrrr)^\aaaa.$$
Hence,
$$|e_{\oooo{n}}|< \dddd{10A}{ (1-\aaaa)2^{\aaaa}}  \oooo{n}^\aaaa h^{3+\aaaa}-\dddd{5A}{ (1-\aaaa) 2^{\aaaa}} \dddd{1-\aaaa}{5}\llll(\dfrac{ 2}{ \oooo{n}}\rrrr)^\aaaa   \oooo{n}^\aaaa h^{3+\aaaa}+A h^{3+\aaaa},$$
$$|e_{\oooo{n}}|< \dddd{10A}{(1-\aaaa) 2^{\aaaa}}  \oooo{n}^\aaaa h^{3+\aaaa}-A h^{3+\aaaa}+A h^{3+\aaaa}=\dddd{10A}{(1-\aaaa) 2^{\aaaa}} \oooo{n}^\aaaa h^{3+\aaaa}.$$
	\end{proof}
	In the next theorem we determine the order of approximation \eqref{3rdOrder} on the interval $[0,T]$.
\begin{thm}\label{t14}
	Suppose that the solution of equation \eqref{eq2} satisfies 
	$$y(0)=y'(0)=y''(0)=y'''(0)=0.$$
	Then approximation \eqref{3rdOrder} converges to the solution with accuracy $O\llll( h^3 \rrrr)$.
	\end{thm}
	\begin{proof} The point $x_n=n h$ is in the interval $[0,T]$ when
	$$n\leq N=\dddd{T}{h}$$
because $T=N h$. From Lemma \ref{ll13},
	$$|e_n|<\dddd{10A}{(1-\aaaa) 2^{\aaaa}} n^\aaaa h^{3+\aaaa}\leq \dddd{10A}{(1-\aaaa) 2^{\aaaa}} \dddd{T^\aaaa}{h^\aaaa} h^{3+\aaaa},$$
		$$|e_n|< \llll(\dddd{10A T^\aaaa}{(1-\aaaa) 2^{\aaaa}} \rrrr) h^{3}.$$
	\end{proof}
	\begin{lem} Suppose that equation \eqref{eq2} has a sufficiently smooth solution. 
	$$L_3=\lim_{x\downarrow 0}\llll( \dddd{d^2}{dx^2}f^{(1-\aaaa)}(x)-
	\dfrac{L_1 x^{\aaaa-2}}{\Gamma (\aaaa)}-
	\dfrac{L_2 x^{\aaaa-1}}{\Gamma (\aaaa-1)}\rrrr).$$
\end{lem}
\begin{proof} By differentiating equation \eqref{eq2} we obtain
$$y''(x)+y^{(1+(1-\aaaa))}(x)=f^{(1+(1-\aaaa))}(x),$$
$$y'''(x)+y^{(2+(1-\aaaa))}(x)=f^{(2+(1-\aaaa))}(x).$$
From \eqref{form27},
$$y^{(2+(1-\aaaa))}(x)=\dfrac{d}{d x}y^{(1+(1-\aaaa))}(x)=
\dfrac{d}{d x}\llll(\dddd{y'(0)x^{\aaaa-1}}{\Gamma(\aaaa)}+\dddd{1}{\Gamma(\aaaa)}  \int_0^x
 \dddd{y''(\xi)}{(x-\xi)^{1-\aaaa}}d\xi\rrrr).$$
Using integration by parts we obtain
$$y^{(2+(1-\aaaa))}(x)=\dddd{y'(0)x^{\aaaa-2}}{\Gamma(\aaaa-1)}+
\dddd{y''(0)x^{\aaaa-1}}{\Gamma(\aaaa)}+\dddd{1}{\Gamma(\aaaa)}  \int_0^x \dddd{y'''(\xi)}{(x-\xi)^{1-\aaaa}}d\xi,$$
$$y^{(2+(1-\aaaa))}(x)=
\dddd{L_2 x^{\aaaa-1}}{\Gamma(\aaaa)}+
\dddd{L_1x^{\aaaa-2}}{\Gamma(\aaaa-1)}+
D_x^{3-\aaaa}y(x).$$
Then
$$y'''(x)=f^{(2+(1-\aaaa))}(x)-\dddd{L_2 x^{\aaaa-1}}{\Gamma(\aaaa)}-
\dddd{L_1x^{\aaaa-2}}{\Gamma(\aaaa-1)}-D_x^{3-\aaaa}y(x).$$
The value of $D_x^{3-\aaaa}y(0)$ is zero, when $y(x)$ is sufficiently differentiable function.
$$y'''(0)=\lim_{x\downarrow 0}\llll( \dddd{d^2}{dx^2}f^{(1-\aaaa)}(x)-\dddd{L_2 x^{\aaaa-1}}{\Gamma(\aaaa)}-
\dddd{L_1x^{\aaaa-2}}{\Gamma(\aaaa-1)}\rrrr).$$
\end{proof}
Let
$$z(x)=y(x)-L_1x-\dddd{L_2}{2} x^2-\dddd{L_3}{6} x^3.$$
\begin{lem} \label{Lem16} The function $z(x)$ satisfies 
$$z(0)=z'(0)=z''(0)=z'''(0)=0$$
 and is a solution of the ordinary fractional differential equation
\begin{align}\label{l16}
z^{(\aaaa)}(x)+z(x)=F(x),
\end{align}
where
$$F(x)=f(x)-L_1 x-\dddd{L_2}{2} x^2-\dddd{L_3}{6} x^3-\dfrac{L_1 x^{1-\aaaa}}{\Gamma (2-\aaaa)}-\dfrac{ L_2 x^{2-\aaaa}}{\Gamma (3-\aaaa)}
-\dfrac{ L_3 x^{3-\aaaa}}{\Gamma (4-\aaaa)}.$$
\end{lem}
From Theorem \ref{t14}, we can compute a third order numerical  solution of equation \eqref{l16} with recurrence relations \eqref{3rdOrder}.
\subsection{Numerical examples}
In Lemma 8, we showed that when equation (2) has a sufficiently differentiable solution the function $f(x)$ satisfies the following conditions.
\begin{itemize}
	\item[(C1)] $f(0)=0;$
	\item[(C2)] The following limits exist
	$$\lim_{x\downarrow 0}f^{(1-\aaaa)}(x)=L_1,\quad 
\lim_{x\downarrow 0}\llll( \dddd{d}{dx}f^{(1-\aaaa)}(x)-
	\dfrac{L_1 x^{\aaaa-1}}{\Gamma (\aaaa)}\rrrr)=L_2.$$
\end{itemize}
When  $L_1=L_2=0$ the solution $y(x)$ of  (2) satisfies $y'(0)=y''(0)=0$. The value of $L_3=y'''(0)$ is determined from the following limit.
$$\lim_{x\downarrow 0}\llll( \dddd{d^2}{dx^2}f^{(1-\aaaa)}(x)-
	\dfrac{L_1 x^{\aaaa-2}}{\Gamma (\aaaa)}-
	\dfrac{L_2 x^{\aaaa-1}}{\Gamma (\aaaa-1)}\rrrr)=L_3.$$
	In Example 1 we use the algorithms from Lemma 9 and Lemma \ref{Lem16} to compute second and third order numerical solutions for equation (2) with right-hand side \eqref{fx}. In Example 2 we consider equations \eqref{eeee1} and \eqref{eeee2} for which the conditions (C1) and (C2) are not satisfied.
\subsubsection{Example 1}
Consider the  ordinary fractional differential  equation 
	\begin{equation*}
	\left\{
	\begin{array}{l l}
	y^{(\alpha )}(x)+y(x)=f(x),&  \\
	y(0)=0.&  \\
	\end{array} \label{eq441}
		\right . 
	\end{equation*}
with right-hand side
\begin{align}\label{fx}
f(x)=2x+3x^2+4x^3+6x^{3+\aaaa}+\dddd{2x^{1-\aaaa}}{\Gamma(2-\aaaa)}&+ \dddd{6x^{2-\aaaa}}{\Gamma(3-\aaaa)}+\\
&\dddd{24x^{3-\aaaa}}{\Gamma(4-\aaaa)}+\Gamma(4+\aaaa)x^3.\nonumber
\end{align}
The function $f(x)$ satisfies $f(0)=0$,
\begin{equation*}
\begin{aligned}
f^{(1-\aaaa)}(x)=\dfrac{2x^{\aaaa}}{\Gamma(1+\aaaa)}&+\dddd{6x^{1+\aaaa}}{\Gamma{(2+\aaaa)}}+\dddd{24x^{2+\aaaa}}{\Gamma{(3+\aaaa)}}+\\
&\dddd{\Gamma(4+\aaaa) x^{2+2\aaaa}}{\Gamma(3+2\aaaa)}+2+
6x+12x^2+6(3+\aaaa)x^{2+\aaaa}.
\end{aligned}
\end{equation*}
Then $L_1=f^{(1-\aaaa)}(0)=2$. Now we compute the value of $L_2=y''(0)$.
\begin{equation*}
\begin{aligned}
\dddd{d}{dx}f^{(1-\aaaa)}(x)=\dfrac{2x^{\aaaa-1}}{\Gamma(\aaaa)}+&\dddd{6x^{\aaaa}}{\Gamma{(1+\aaaa)}}+\dddd{24x^{1+\aaaa}}{\Gamma{(2+\aaaa)}}+\dddd{\Gamma(4+\aaaa) x^{1+2\aaaa}}{\Gamma(2+2\aaaa)}\\
&+6+24x+6(2+\aaaa)(3+\aaaa)x^{1+\aaaa}
\end{aligned}
\end{equation*}
$$L_2=\lim_{x\downarrow 0}\llll(\dddd{d}{dx}f^{(1-\aaaa)}(x)-\dddd{2 x^{\aaaa-1}}{\Gamma(\aaaa)} \rrrr)=6.$$
Let
$$F_1(x)=f(x)-L_1 x-\dddd{L_2}{2} x^2-\dfrac{L_1}{\Gamma (2-\aaaa)}x^{1-\aaaa}-\dfrac{ L_2}{\Gamma (3-\aaaa)}x^{2-\aaaa},$$
$$F_1(x)=4x^3+6x^{3+\aaaa}+\dddd{24x^{3-\aaaa}}{\Gamma(4-\aaaa)}+\Gamma(4+\aaaa)x^3.$$
The function $z_1(x)=y(x)-2x-3x^2$ is a solution of
\begin{equation}\label{eq4411}
z_1^{(\aaaa)}(x)+z_1(x)=4x^3+6x^{3+\aaaa}+\dddd{24x^{3-\aaaa}}{\Gamma(4-\aaaa)}+\Gamma(4+\aaaa)x^3.
\end{equation}
The  solution of equation \eqref{eq4411} is $z_1(x)=4x^3+6x^{3+\aaaa}$. From Theorem 7, we can compute second order numerical solutions of equation \eqref{eq4411} with approximations \eqref{Recurrence3} and \eqref{Recurrence4}. Experimental results for aprroximation \eqref{Recurrence4} and $\aaaa=0.25$ are given in Table 3 and Figure 2. 
\begin{equation*}
\begin{aligned}
\dddd{d^2}{d x^2}f^{(1-\aaaa)}(x)=\dddd{2x^{\aaaa-1}}{\Gamma{(\aaaa-2)}}+\dddd{6x^{\aaaa-1}}{\Gamma{(\aaaa)}}+&\dddd{24x^{\aaaa}}{\Gamma{(1+\aaaa)}}+\dddd{\Gamma(4+\aaaa) x^{2\aaaa}}{\Gamma(1+2\aaaa)}+\\
&24+6(\aaaa+1)(\aaaa+2)(\aaaa+3)x^{\aaaa},
\end{aligned}
\end{equation*}
$$L_3=\lim_{x\rightarrow 0}\llll(\dddd{d^2}{d x^2}f^{(1-\aaaa)}(x)-\dddd{2 x^{\aaaa-2}}{\Gamma(\aaaa)} -\dddd{6 x^{\aaaa-1}}{\Gamma(\aaaa-1)} \rrrr)=24.$$
Let
$$F_2(x)=f(x)-L_0-L_1 x-\dddd{L_2}{2} x^2-\dddd{L_3}{6} x^3-\dfrac{L_1 x^{1-\aaaa}}{\Gamma (2-\aaaa)}-\dfrac{ L_2 x^{2-\aaaa}}{\Gamma (3-\aaaa)}
-\dfrac{ L_3 x^{3-\aaaa}}{\Gamma (4-\aaaa)},$$
$$F_2(x)=6x^{3+\aaaa}+\Gamma(4+\aaaa)x^3.$$
The function $z_2(x)=y(x)-1-2x-3x^2-4x^3$ is a solution of the equation
\begin{equation}\label{eq4412}
z_2^{(\aaaa)}(x)+z_2(x)=6x^{3+\aaaa}+\Gamma(4+\aaaa)x^3.
\end{equation}
Equation \eqref{eq4412} has solution $z_2(x)=6x^{3+\aaaa}$. Experimental results for a third order numerical solution of equation \eqref{eq4412} using recurrence relations \eqref{3rdOrder} are given in Table 4 and Figure 3.
\begin{figure}
\centering
\begin{minipage}{.45\textwidth}
  \centering
  \includegraphics[width=.95\linewidth]{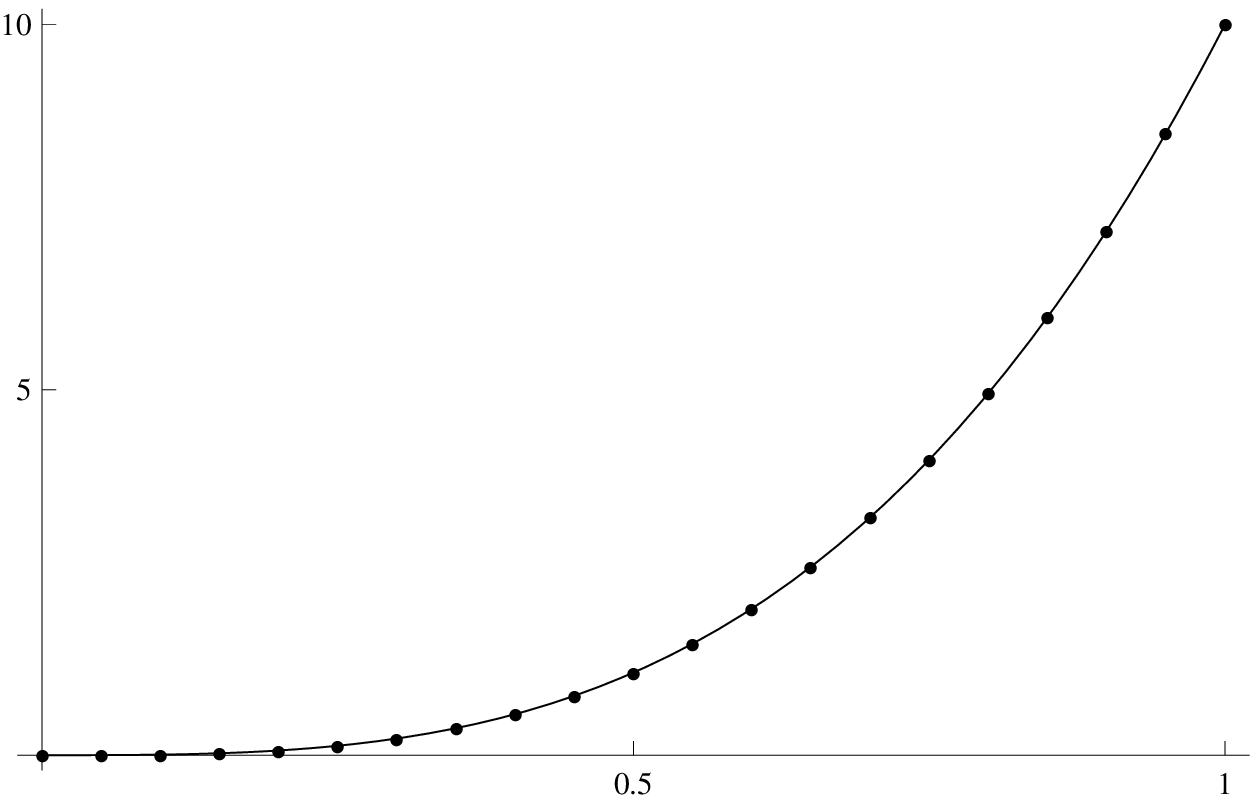}
  \captionof{figure}{Graph of the  solution of equation \eqref{eq4411} and  second order approximation \eqref{Recurrence4} for $h=0.05$ and $\alpha=0.25$.} 
  \label{fig:test1}
\end{minipage}%
\hspace{0.5cm}
\begin{minipage}{.45\textwidth}
  \centering
  \includegraphics[width=.95\linewidth]{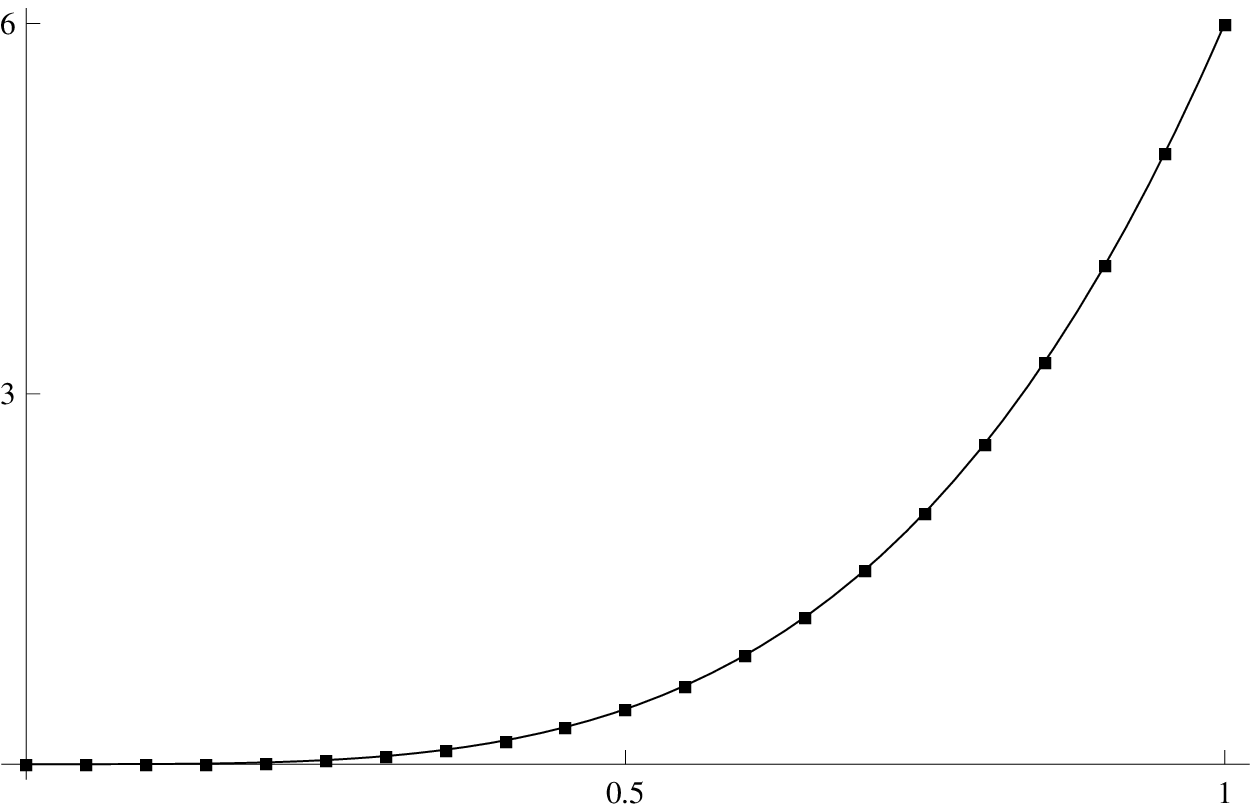}
  \captionof{figure} {Graph of the solution of equation \eqref{eq4412} and  third order approximation \eqref{3rdOrder} for $h=0.05$ and $\alpha=0.75$.}
  \label{fig:test2}
\end{minipage}
\end{figure}
	\begin{table}
	\caption{Maximum error and order of approximation \eqref{Recurrence4} for equation \eqref{eq4411} with  $\alpha=1/4$.}
	\centering
  \begin{tabular}{ l  c c c }
    \hline \hline
    $h$  & $Error$                 & $Ratio$ & $Order$  
		\\ \hline \hline
$0.05$     & $0.00393827$          & $3.93462$  & $1.97622$   \\ 
$0.025$    & $0.00099263$         & $3.96753$  & $1.98824$   \\ 
$0.0125$   & $0.00024916$         & $3.98382$  & $1.99415$   \\ 
$0.00625$  & $0.00006242$        & $3.99192$  & $1.99708$   \\ 
$0.003125$ & $0.00001562$          & $3.99596$  & $1.99854$   \\ \hline
  \end{tabular}
	\end{table}
	\begin{table}
	\caption{Maximum error and order of approximation \eqref{3rdOrder} for equation \eqref{eq4412} with   $\alpha=3/4$.}
	\centering
  \begin{tabular}{ l  c c c }
    \hline \hline
    $h$  & $Error$                 & $Ratio$ & $Order$  
		\\ \hline \hline
$0.05$     & $0.000314897$                 & $7.52526$  & $2.91174$   \\ 
$0.025$    & $0.000040446$                & $7.78566$  & $2.96082$   \\ 
$0.0125$   & $5.121\times 10^{-6}$       & $7.89840 $  & $2.98156$   \\ 
$0.00625$  & $6.441\times 10^{-7}$        & $7.95061$  & $2.99107$   \\ 
$0.003125$ & $8.075\times 10^{-8}$       & $7.97568$  & $2.99561$   \\ \hline
  \end{tabular}
	\end{table}
	\subsubsection{Example 2}
	We compute numerical solutions for two ordinary fractional differential equations  for which the conditions (C1) and (C2) for differentiable solution are not satisfied.
	\begin{equation}\label{eeee1}
	y^{(0.25)}(x)+y(x)=x^{0.25}+\Gamma(1.25).
	\end{equation}
	The solution of equation \eqref{eeee1} is $y(x)=x^{0.25}$. The solution is not a continuously differentiable function, because condition (C1) is not satisfied
	$$f(0)=\Gamma(1.25)\neq 0.$$
	We compute a numerical solution of equation \eqref{eeee1} with recurrence relations \eqref{Recurrence3}. When $h=0.025$ the error is $0.0167995$. The approximation  converges to the solution with a  very slow rate and the approximation order is smaller than one. Experimental results are given in Table 5.
		\begin{table}
	\caption{Maximum error and order of approximation  \eqref{Recurrence3} for  equation \eqref{eeee1}.}
	\centering
  \begin{tabular}{ l  c c c }
    \hline \hline
    $h$  & $Error$          & $Ratio$    & $\log_2 (Ratio)$  
		\\ \hline \hline 
$0.0125$     & $0.0162401$    & $1.03445$  & $0.048863$   \\ 
$0.00625$    & $0.0152728$    & $1.06333$  & $0.088589$   \\ 
$0.003125$   & $0.0140672$    & $1.08571$  & $0.118636$   \\ 
$0.0015625$  & $0.0127479$    & $1.10349$  & $0.142075$   \\ 
$0.00078125$ & $0.0114037$    & $1.11787$  & $0.160754$   \\ 
\hline
\end{tabular}
	\end{table}
	
		The following equation has solution $y(x)=x^{1.25}$.
	\begin{equation}\label{eeee2}
	y^{(0.25)}(x)+y(x)=\Gamma(1.25)x+x^{1.25}.
	\end{equation}
	The solution is equation \eqref{eeee2} has better differentiability properties than the solution of equation \eqref{eeee1}. We can expect that approximation  \eqref{Recurrence3} has higher accuracy for equation \eqref{eeee2}.
	$$f(x)=\Gamma(1.25)x+x^{1.25}, \qquad (f(0)=0).$$
	$$f^{(0.75)}(x)=x^{0.25}+\dddd{\Gamma(2.25)}{\Gamma(1.5)}x^{0.5},
	\quad L_1=\lim_{x\downarrow 0} f^{(0.75)}(x)=0.
	$$
		$$\dddd{d}{dx}f^{(0.75)}(x)=\dddd{0.25}{x^{0.5}}+\dddd{\Gamma(2.25)}{\Gamma(0.5)x^{0.5}},
	\quad L_2=\lim_{x\downarrow 0} \dddd{d}{dx}f^{(0.75)}(x)=\infty.$$
	The solution of equation \eqref{eeee2} doesn't have a continuous second derivative on the interval $[0,1]$, because condition (C2) is not satisfied. 
	When $h=0.025$ the error is $0.000148$. The accuracy of approximation  \eqref{Recurrence3} for equation \eqref{eeee2} is around $O\llll(h^{1.31}\rrrr)$, when $h>0.0008$ (Table 6).
		\begin{table}
	\caption{Maximum error and order of approximation  \eqref{Recurrence3} for  equation \eqref{eeee1}.}
	\centering
  \begin{tabular}{ l  c  c c }
    \hline \hline
    $h$  & $Error$          & $Ratio$    & $Order$  
		\\ \hline \hline 
$0.0125$     & $0.0000589415$             & $2.50814$  & $1.32662$   \\ 
$0.00625$    & $0.0000235523$             & $2.50259$  & $1.32342$   \\ 
$0.003125$   & $9.4369\times 10^{-6}$     & $2.49576$  & $1.31948$   \\ 
$0.0015625$  & $3.7930\times 10^{-6}$    & $2.48797$  & $1.31497$   \\ 
$0.00078125$ & $1.5297\times 10^{-6}$    & $2.47954$  & $1.31007$   \\ 
\hline
  \end{tabular}
	\end{table}
\section{Second-order implicit difference approximations for the time fractional sub-diffusion equation}
In the present section we determine implicit difference approximations for the fractional sub-diffusion equation. We show that when the solution of the sub-diffusion equation is a sufficiently differentiable function the difference approximations have second order accuracy $O\llll(\tau^2+h^2\rrrr)$.
 The analytic solution  of the fractional sub-diffusion equation can be determined using Laplace-Fourier transform \cite{Podlubny1999} or separation of variables for special cases of the boundary conditions and the function $G(x,t)$.
The fractional diffusion equation 
	\begin{equation}
	\left\{
	\begin{array}{l l}
	\dfrac{\partial^\alpha u(x,t)}{\partial t^\alpha}=\dfrac{\partial^2 u(x,t)}{\partial x^2},& \\
	u(0,t)=u(1,t)=0, u(x,0)=g(x).&  \\
	\end{array} \label{FDE2}
		\right . 
	\end{equation}
 has an analytical solution \cite{ HejaziMoroneyLiu2014}.
\begin{align}\label{exact1}
u(x,t)=2\sum_{n=1}^\infty c_n E_\aaaa (-n^2 \pi^2 t^\aaaa) \sin (n\pi x)
\end{align}
on the domain $\{0\leq x\leq 1,t\geq 0\} $,where
$$c_n=\int_0^1 g(\xi) \sin (n\pi \xi)d\xi$$
and  $E_\aaaa$ is the one-parameter Mittag-Leffler function.
Each term 
$$E_\aaaa (-n^2 \pi^2 t^\aaaa) \sin (n\pi x)$$
 is a solution of \eqref{fdeqn} and its coefficient $c_n$ is the coefficient of the Fourier sine series of the function $g(x)$.
The graph of the analytical  solution of \eqref{FDE2} when $g(x)=x^2(x-1)$ is given in Figure 4. 
\begin{figure}[ht]
  \centering
  \caption{Analytical solution of  \eqref{FDE2} for $\aaaa=1/2,\; g(x)=x^2(x-1)$ and $0\leq t\leq 0.05$.} 
  \includegraphics[width=0.6\textwidth]{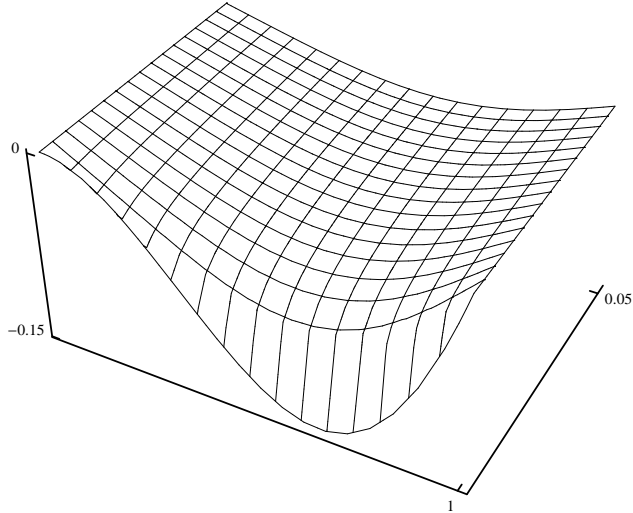}
\end{figure}
\subsection{Second order implicit difference approximation}
The fractional sub-diffusion equation is an important equation in fractional calculus. The implicit difference approximation which uses approximation \eqref{CaputoApproximation} for the fractional derivative and central difference approximation for the second derivative with respect to $x$, has accuracy $O\llll( \tau^{2-\aaaa}+h^2\rrrr)$\cite{GaoSunZhang2012}.
Finite difference approximations are convenient way to approximate the solution of partial fractional differential equations. They combine  simple description with stability and high accuracy. Even when the exact solution of the time-fractional diffusion equation is available, the finite difference approximations may have higher accuracy than approximations using the exact solution. The approximation error of a numerical solution of \eqref{FDE2} computed by truncating \eqref{exact1} includes errors from the truncation of the Fourier series at the endpoints and the approximations of the coefficients $c_n$ and the Mittag-Leffler functions. In this section we determine second order difference approximations for equation (1) on the domain
	$$D=[0,1]\times [0,T].$$
 We can assume that equation  \eqref{fdeqn} has homogeneous initial and  boundary  conditions
	\begin{equation}\label{HomogFDE}
	\left\{
	\begin{array}{l l}
\dfrac{\partial^\alpha u(x,t)}{\partial t^\alpha}=\dfrac{\partial^2 u(x,t)}{\partial x^2}+G(x,t),& \\
	u(0,t)=u(1,t)=0, u(x,0)=0.&  \\
	\end{array} 
		\right . 
	\end{equation}
If equation \eqref{fdeqn} is given with non-homogeneous  initial or boundary conditions the substitution 
$$\overline{u}(x,t)= u(x,t)-u(x,0)-(1-x)(u(0,t)-u(0,0))-x (u(1,t)- u(1,0))$$
converts the equation to equation which has the same form and homogenous initial and boundary conditions. In Lemma \ref{l6}, we showed that when the solution $u(x,t)$ is sufficiently differentiable function with respect to the time variable $t$, the fractional derivative $\dfrac{\partial^\alpha u(x,0)}{\partial t^\alpha}$ is zero when $t=0$.
When the solution $u(x,t)$ has continuous second derivative $u_{xx}(x,t)$  with respect to $t$, the function $G(x,t)$ satisfies the  condition $G(x,0)=0$.
$$G(x,0)=\llll. \dfrac{\partial^\alpha u(x,t)}{\partial t^\alpha}\rrrr|_{t=0}-\llll. \dfrac{\partial^2 u(x,t)}{\partial x^2}\rrrr|_{t=0}=0.$$
This compatibility condition corresponds to the condition $f(0)=0$ for differentiable solution of ordinary fractional differential equation (2), with initial condition $y(0)=0$.
In order to construct  second order difference approximations for equation \eqref{FDE2} using approximations \eqref{SecondOrder} and \eqref{SecondOrder2} for the Caputo derivative, the first step is to ensure that the first and second derivatives of the solution $u_t(x,t)$ and $u_{tt}(x,t)$ with respect to the time variable $t$  are equal to zero when $t=0$. Let
$$ L_1(x)=\llll. \dfrac{\partial u(x,t)}{\partial t}\rrrr|_{t=0},\quad 
L_2(x)=\llll. \frac{\partial^2 u(x,0)}{\partial t^2}\rrrr|_{t=0},\quad
L(x)=\llll. \frac{\partial^3 u(x,0)}{\partial t\partial x^2}\rrrr|_{t=0}.
$$
By applying time fractional derivative of order $1-\aaaa$ to equation (1) we obtain
\begin{align}\label{s41}
\dfrac{\partial u(x,t)}{\partial t}=\dfrac{\partial^{3-\aaaa} u(x,t)}{\partial t^{1-\aaaa} \partial x^2}+\dddd{\partial^{1-\aaaa}G(x,t)}{t^{1-\aaaa}}.
\end{align}
When $u_{xx}(x,t)$ is bounded in $D$, we have that $\frac{\partial^{3-\aaaa} u(x,t)}{\partial t^{1-\aaaa} \partial x^2}=0$. Then
\begin{align}\label{LL1}
L_1(x)=\llll. \dddd{\partial^{1-\aaaa}G(x,t)}{t^{1-\aaaa}}\rrrr|_{t=0}.
\end{align}
 By differentiating \eqref{s41} with respect to $t$ we obtain
 \begin{align}\label{e49}
\dfrac{\partial^2 u(x,t)}{\partial t^2}=
\dfrac{\partial}{\partial t} \dfrac{\partial^{3-\aaaa} u(x,t)}{\partial t^{1-\aaaa} \partial x^2}+\dfrac{\partial}{\partial t}\dddd{\partial^{1-\aaaa}G(x,t)}{t^{1-\aaaa}}.
\end{align} 
The Caputo derivative of order $1-\aaaa$ of the function $u_{xx}(x,t)$ with respect to $t$ is defined as
$$\dfrac{\partial^{3-\aaaa} u(x,t)}{\partial t^{1-\aaaa} \partial x^2}=
\dddd{1}{\Gamma(\aaaa)}\int_{0}^{t} \dfrac{\partial^{3} u(x,\xi)}{\partial t \partial x^2} (t-\xi)^{\aaaa-1}d\xi.$$
After integration by parts and differentiation with respect to $t$ we obtain
$$\dfrac{\partial}{\partial t} \dfrac{\partial^{3-\aaaa} u(x,t)}{\partial t^{1-\aaaa} \partial x^2}=\dddd{t^{\aaaa-1}}{\Gamma(\aaaa)} \llll.\dddd{\partial^3 u(x,t)}{\partial t\partial x^2}\rrrr|_{t=0}+\dddd{1}{\Gamma(\aaaa)}
\int_0^t \dddd{\partial^4 u(x,\xi)}{\partial t^2 \partial x^2}(t-\xi)^{\aaaa-1}d\xi.$$
From the definition of Caputo derivative of order $2-\aaaa$ with respect to $t$
$$\dddd{1}{\Gamma(\aaaa)}
\int_0^t \dddd{\partial^4 u(x,\xi)}{\partial t^2 \partial x^2}(t-\xi)^{\aaaa-1}d\xi=D_t^{2-\aaaa} \dddd{\partial^2 u(x,t)}{ \partial x^2}.$$
So,
\begin{align}\label{e48}
\dfrac{\partial}{\partial t} \dfrac{\partial^{3-\aaaa} u(x,t)}{\partial t^{1-\aaaa} \partial x^2}=\dddd{t^{\aaaa-1}}{\Gamma(\aaaa)} L(x)+D_t^{2-\aaaa} \dddd{\partial^2 u(x,t)}{ \partial x^2}.
\end{align}
We have that
$$\dddd{\partial^3 u(x,t)}{\partial t\partial x^2}=
\dddd{\partial }{\partial t}\llll( \dddd{\pppp^\aaaa u(x,t)}{\pppp t^\aaaa}-G(x,t) \rrrr)=
\dddd{\partial }{\partial t} \dddd{\pppp^\aaaa u(x,t)}{\pppp t^\aaaa}-\dddd{\partial }{\partial t}G(x,t).$$
By integrating by parts and differentiating with respect to $t$
$$\dddd{\partial }{\partial t} \dddd{\pppp^\aaaa u(x,t)}{\pppp t^\aaaa}=\dddd{\partial u(x,0)}{\partial t} \dddd{t^{-\aaaa}}{\Gamma(1-\aaaa)}+
\dddd{1}{\Gamma(1-\aaaa)} \int_0^t \dddd{\partial^2 u(x,\xi)}{\partial t^2}(t-\xi)^{\aaaa}d\xi.$$
Then
$$\dddd{\partial^3 u(x,t)}{\partial t\partial x^2}=
\dddd{\partial u(x,0)}{\partial t} \dddd{t^{-\aaaa}}{\Gamma(1-\aaaa)}+
\dddd{1}{\Gamma(1-\aaaa)} \int_0^t \dddd{\partial^2 u(x,\xi)}{\partial t^2}
(t-\xi)^{\aaaa}d\xi-\dddd{\partial }{\partial t}G(x,t), $$
$$\dddd{\partial^3 u(x,t)}{\partial t\partial x^2}=
 \dddd{L_1(x)}{\Gamma(1-\aaaa)t^\aaaa}+
D_t^{1+\aaaa}  u(x,t)-\dddd{\partial }{\partial t}G(x,t). $$
The function $L(x)$ is computed as
\begin{align}\label{LL}
L(x)=\llll. \dddd{\partial^3 u(x,t)}{\partial t\partial x^2}\rrrr|_{t=0}=
\lim_{t\downarrow 0} \llll( \dddd{L_1 (x)}{\Gamma(1-\aaaa)t^\aaaa}-\dddd{\partial }{\partial t}G(x,t) \rrrr)
\end{align}
because the Caputo derivative $D_t^{1+\aaaa}  u(x,0)$ is zero when $u_{tt}(x,t)$ is bounded. The function $L_2(x)$ is computed from \eqref{e49} and \eqref{e48}
 $$\dfrac{\partial^2 u(x,t)}{\partial t^2}=
\dddd{t^{\aaaa-1}}{\Gamma(\aaaa)} L(x)+D_t^{2-\aaaa} \dddd{\partial^2 u(x,t)}{ \partial x^2}+\dfrac{\partial}{\partial t}\dddd{\partial^{1-\aaaa}G(x,t)}{t^{1-\aaaa}},$$ 
 \begin{align}\label{LL2}
L_2(x)=\dfrac{\partial^2 u(x,0)}{\partial t^2}=\lim_{t\downarrow 0}\llll(
\dddd{L(x)}{\Gamma(\aaaa)t^{1-\aaaa}}+\dfrac{\partial}{\partial t}\dddd{\partial^{1-\aaaa}G(x,t)}{t^{1-\aaaa}}\rrrr).
\end{align}
Let
$$v(x,t)=u(x,t)-L_1(x)t-\dfrac{L_2(x)}{2}t^2.$$
The partial derivatives $v_t(x,t)$ and $v_{tt}(x,t)$ of the function $v(x,t)$ are equal to zero when $t=0$.
$$\dddd{\partial^\alpha v(x,t)}{\partial t^\alpha}=
\dddd{\partial^\alpha u(x,t)}{\partial t^\alpha}-\dddd{L_1(x)t^{1-\aaaa}}{\Gamma(2-\aaaa)}-\dddd{L_2(x)t^{2-\aaaa}}{\Gamma(3-\aaaa)},$$
$$\dddd{\partial^2 v(x,t)}{\partial x^2}=
\dddd{\partial^2 u(x,t)}{\partial x^2}- L''_1(x)t-\dfrac{L''_2(x)}{2}t^2.$$
The function $v(x,t)$ is solution of the fractional sub-diffusion equation
	\begin{equation}\label{VHomogFDE}
	\left\{
	\begin{array}{l l}
\dfrac{\partial^\alpha v(x,t)}{\partial t^\alpha}=\dfrac{\partial^2 v(x,t)}{\partial x^2}+H(x,t),& \\
	v(0,t)=v(1,t)=0, v(x,0)=0,&  \\
	\end{array} 
		\right . 
	\end{equation}
where
\begin{align}\label{H}
H(x,t)=G(x,t)-\dddd{L_1(x)t^{1-\aaaa}}{\Gamma(2-\aaaa)}-\dddd{L_2(x)t^{2-\aaaa}}{\Gamma(3-\aaaa)}+L''_1(x)t+\dfrac{L''_2(x)}{2}t^2.
\end{align}
In section 4 we used approximation \eqref{SecondOrder2} to determine second order numerical solutions \eqref{Recurrence3} and \eqref{Recurrence4} for equation \eqref{eq2}. Now we use  \eqref{SecondOrder2} and a central difference approximation for $u_{xx}(x,t)$ to construct implicit difference approximations \eqref{FiniteDifferenceScheme} and \eqref{FiniteDifferenceScheme2} for equation \eqref{VHomogFDE}. In Theorem \ref{t29} we show that the difference approximations are unconditionally stable and have second order accuracy $O\llll(\tttt^2+h^2\rrrr)$.
Let $h=1/N$ and $\tttt=T/M$ where $M$ and $N$ are   positive  integers, and
$$x_n=nh,\quad t_m=m\tau,\quad v_n^m=v(x_n,t_m),\quad H_n^m=H(x_n,t_m).$$
From approximation \eqref{SecondOrder2} and equation \eqref{VHomogFDE}
\begin{align*}
\Delta_h^\aaaa v(x_n, t_m )&=\llll( \dddd{\aaaa}{2}\rrrr)
\dddd{\partial^\aaaa v(x_n, t_{m-1} )}{\partial t^\aaaa}+
\llll( 1-\dddd{\aaaa}{2}\rrrr)
\dddd{\partial^\aaaa v(x_n, t_{m} )}{\partial t^\aaaa}+O\left(\tau^2\right)\\
&=\llll( \dddd{\aaaa}{2}\rrrr)
\dddd{\partial^2 v(x_n, t_{m-1} )}{\partial x^2}+
\llll( 1-\dddd{\aaaa}{2}\rrrr)
\dddd{\partial^2 v(x_n, t_{m} )}{\partial x^2}\\
&+\llll( \dddd{\aaaa}{2}\rrrr)H_n^{m-1}+
\llll( 1-\dddd{\aaaa}{2}\rrrr)H_n^m+O\left(\tau^2\right).
\end{align*}
By approximating the second derivatives $u_{xx}(x_n,t_{m-1})$ and $u_{xx}(x_n,t_{m})$ with central difference formulas we obtain
\begin{align*}
 \dfrac{1}{\tau^\alpha} \sum_{k=0}^m & w_k^{(\alpha)} v_n^{m-k}+O\left(\tttt^2+h^2\right) =\left(1-\dfrac{\alpha}{2}\right)
\dfrac{v_{n-1}^{m}-2v_n^{m}+v_{n+1}^{m}}{h^2}\\
&+\left(\dfrac{\alpha}{2}\right) \dfrac{v_{n-1}^{m-1}-2v_n^{m-1}+v_{n+1}^{m-1}}{h^2}+\llll( \dddd{\aaaa}{2}\rrrr)H_n^{m-1}+
\llll( 1-\dddd{\aaaa}{2}\rrrr)H_n^m.
\end{align*}
Let $\eta=\frac{\tau^\alpha }{h^2}$. The solution of equation \eqref{VHomogFDE} satisfies
\begin{align}
& v_n^{m}-\llll(1-\frac{\alpha}{2} \rrrr) \eta  \left(v_{n-1}^{m}-2v_n^{m}+v_{n+1}^{m}\right)+\tau^\alpha O\left(\tau^2+h^2\right) =-  \sum_{k=2}^m w_k^{(\alpha)} v_n^{m-k} \nonumber\\
&+\alpha v_n^{m-1}+\frac{\alpha \eta}{2}  \left(v_{n-1}^{m-1}-2v_n^{m-1}+v_{n+1}^{m-1}\right)+\tau^\alpha \llll( \llll( \dddd{\aaaa}{2}\rrrr)H_n^{m-1}+
\llll( 1-\dddd{\aaaa}{2}\rrrr)H_n^m\rrrr).\label{eqnn44}
\end{align}
Let $x=(x_n)$ be an $N-1$ dimensional vector. The maximum (infinity) norm  of the vector $x$ is 
$$\llll\| x\rrrr\|=\max_{1\leq n\leq N-1} |x_n|.$$
Define the vectors $V_m,H_m$ and $K_m$ as
\begin{itemize}
	\item $V_m=\llll( v_n^m\rrrr)_{n=1}^{N-1}$ - a vector of  values of the exact solution at time $t=m\tau$;
	\item $H_m=\llll(  \llll( \dddd{\aaaa}{2}\rrrr)H_n^{m-1}+
\llll( 1-\dddd{\aaaa}{2}\rrrr)H_n^m\rrrr)_{n=1}^{N-1};$
	\item $K_m=\llll( k_n^m\rrrr)_{n=1}^{N-1}$ - a vector of the truncation errors at $t=m\tau$. In \eqref{eqnn44} we showed that $\llll\| K^m\rrrr\|\in  O\left(\tau^\alpha\llll(\tau^2+h^2\right)\rrrr)$. The elements of the vector $K_m$ satisfy
	$$\llll|k^m_n\rrrr|<K \tau^\alpha\llll(\tau^2+h^2\right),$$
\end{itemize}
	where $K>1$ is a positive constant (The conditions $K>1$  and  \eqref{CR} guarantee that $C_R>1$).
	
Let $g_n^{(\alpha)}=-w_n^{(\alpha)}$ and  $A=A_{N-1}$ be a tridiagonal square matrix with entries $2$ on the main diagonal and $-1$ on the first diagonals below and above the main diagonal.
$$A_5 =
 \begin{pmatrix}
  2       &  -1       & 0        & 0          & 0       \\
 -1       &  2        & -1       & 0          & 0       \\
  0       & -1        &  2       & -1         & 0       \\
  0       & 0         & -1       &  2         & -1 \\
	0       & 0         & 0        & -1         &  2   
 \end{pmatrix}
,\qquad
X_5 =
 \begin{pmatrix}
b      & a      &  0       & 0          & 0       \\
c      & b      &  a       & 0          & 0       \\
0      & c      &  b       & a          & 0       \\
0      & 0      &  c       &  b         & a \\
0      & 0      &  0       & c          & b   
 \end{pmatrix}.
$$
The numbers $g_n^{(\aaaa)}$ are positive for $n\geq1$ and 
$\sum_{n=1}^\infty g_n^{(\aaaa)}=1$.
The eigenvalues of the matrix $A$ are determined from the following more general result for eigenvalues of a  tridiagonal matrix \cite{DingLi2013}.
\begin{lem}  The eigenvalues of the tridiagonal matrix $X=X_{N-1}$ with entries $b$ on the main diagonal and $a$ and $c$ on the first diagonals above and below the main diagonal
are given by
$$\lambda_k=b+2a\sqrt{\dddd{c}{a}}\cos\llll(\dddd{k \pi }{N} \rrrr),\quad (k=1,2,\cdots,N-1).$$
\end{lem}
\begin{cor}\label{c18}
The matrix $A$ has eigenvalues
$$\lambda_k=4 \sin^2 \llll( \dddd{k \pi}{N}\rrrr),\quad (k=1,\cdots,N-1).$$
\end{cor}
\begin{proof} From Lemma 5 with $a=c=-1$ and $b=2$ we obtain
$$\lambda_k=2-2\cos \llll( \dddd{k \pi}{N}\rrrr)=4 \sin^2 \llll( \dddd{k \pi}{N}\rrrr).$$
\end{proof}
Equation \eqref{eqnn44} is written in a matrix form as
$$\llll(I+\llll(1-\frac{\alpha}{2} \rrrr)\eta A\rrrr) V_m=\llll(\aaaa I-\frac{\alpha \eta}{2}  A\rrrr) V_{m-1}+\sum_{k=2}^{m-1} g_k^{(\alpha)} V_{m-k}+\tau^\alpha H_m+K_m.$$
Let $P$ and $Q$ be the following matrices.
$$P=I+\llll(1-\frac{\alpha}{2} \rrrr)\eta A, \quad Q=I-\dfrac{\eta}{2} A.$$
Then
$$P V_m=\aaaa Q V_{m-1}+\sum_{k=2}^{m-1} g_k^{(\alpha)} V_{m-k}+
\tau^\alpha H_m+K_m.$$
We compute an approximation $\wwww{V}_m$ to the exact solution 
$V_m$ of equation \eqref{VHomogFDE} at time $t_m=m\tttt$ on the grid
$$\{(x_n,t_m)|1\leq n\leq N,1\leq m\leq M\}$$
  with $\wwww{V}_0=0$ and the linear systems
\begin{equation}\label{FiniteDifferenceScheme}
P \wwww{V}_m=\aaaa Q \wwww{V}_{m-1}+\sum_{k=2}^{m-1} g_k^{(\alpha)} \wwww{V}_{m-k}+\tau^\alpha H_m.
\end{equation}
The values of the elements of the vector $H_m$ satisfy 
$$\llll( \dddd{\aaaa}{2} \rrrr)H_n^{m-1}+
\llll( 1-\dddd{\aaaa}{2}\rrrr)H_n^m=H_n^{m-\aaaa/2}+O\llll(\tau^2\rrrr).$$
Another  difference  approximation for equation \eqref{VHomogFDE} is computed recursively with  the linear systems
\begin{equation}\label{FiniteDifferenceScheme2}
P \wwww{V}_m=\aaaa Q \wwww{V}_{m-1}+\sum_{k=2}^{m-1} g_k^{(\alpha)} \wwww{V}_{m-k}+\tau^\alpha \oooo{H}_m,
\end{equation}
where $\oooo{H}_m=\llll(H_n^{m-\aaaa/2} \rrrr)_{n=1}^{N-1}$.
In Theorem \ref{t29} we show that the implicit difference  approximations \eqref{FiniteDifferenceScheme} and \eqref{FiniteDifferenceScheme2} are unconditionally stable and have second order accuracy with respect to the  space and time variables. The proof relies on the positivity of the eigenvalues of the matrix $A$ and the lower bound \eqref{Estimate} for sums of Gr\"unwald weights.
\subsection{Numerical example}
We compute numerical solutions of the fractional sub-diffusion equation (1) with homogeneous initial and boundary conditions. The difference approximations \eqref{FiniteDifferenceScheme} and \eqref{FiniteDifferenceScheme2} have comparable properties. In some experiments the truncation error of \eqref{FiniteDifferenceScheme2} is smaller than the truncation error of \eqref{FiniteDifferenceScheme}, and it has slightly better overall performance. When
$$G(x,t)=\dddd{a(x)t^{1-\aaaa}}{\Gamma(2-\aaaa)}+\dddd{2b(x)t^{2-\aaaa}}{\Gamma(3-\aaaa)}+\dddd{\Gamma(3+\aaaa)c(x)t^{2}}{2}-a''(x)t-b''(x)t^2-c''(x)t^{2+\aaaa}$$
equation \eqref{HomogFDE} has solution
$$u(x,t)=a(x)t+b(x)t^2+c(x)t^{2+\aaaa}.$$
The first step is to compute the functions $L_1(x),L(x)$ and $L_2(x)$.
$$L_1(x)=\llll. \dddd{\partial^{1-\aaaa}G(x,t)}{t^{1-\aaaa}}\rrrr|_{t=0},\quad L(x)=\lim_{t\downarrow 0} \llll( \dddd{L_1 (x)}{\Gamma(1-\aaaa)t^\aaaa}-\dddd{\partial }{\partial t}G(x,t) \rrrr),$$
 $$L_2(x)=\lim_{t\downarrow 0}\llll(
\dddd{L(x)}{\Gamma(\aaaa)t^{1-\aaaa}}+\dfrac{\partial}{\partial t}\dddd{\partial^{1-\aaaa}G(x,t)}{t^{1-\aaaa}}\rrrr),$$
$$\dddd{\partial^{1-\aaaa}G(x,t)}{\partial t^{1-\aaaa}}=a(x)+2b(x)t+(2+\aaaa)c(x)t^{1+\aaaa}-\dfrac{a''(x)t^\aaaa}{\Gamma(1+\aaaa)}-\dfrac{2b''(x)t^{1+\aaaa}}{\Gamma(2+\aaaa)}-\dfrac{\Gamma(3+\aaaa)c''(x)t^{1+2\aaaa}}{\Gamma(2+2\aaaa)}.$$
By setting $t=0$ we obtain
$$L_1(x)=a(x),$$
$$\dfrac{\pppp}{\pppp t}G(x,t)=\dddd{a(x)t^{-\aaaa}}{\Gamma(1-\aaaa)}+
\dddd{2b(x)t^{1-\aaaa}}{\Gamma(2-\aaaa)}+\Gamma(3+\aaaa) c(x)t-a''(x)-2b''(x)t-(2+\aaaa)c''(x)t^{1+\aaaa},$$
$$L(x)=\lim_{t\downarrow 0} \llll( \dddd{a(x)}{\Gamma(1-\aaaa)t^\aaaa}-\dddd{\partial }{\partial t}G(x,t) \rrrr)=a''(x),$$
$$\dddd{\pppp}{\pppp t}\dddd{\partial^{1-\aaaa}G(x,t)}{\partial t^{1-\aaaa}}=
2b(x)+(1+\aaaa)(2+\aaaa)c(x)t^\aaaa-\dddd{a''(x)t^{\aaaa-1}}{\Gamma(\aaaa)}-\dddd{2b''(x)t^{\aaaa}}{\Gamma(1+\aaaa)}-
\dddd{\Gamma(3+\aaaa)c''(x)t^{2\aaaa}}{\Gamma(1+2\aaaa)},$$
$$L_2(x)=\lim_{t\downarrow 0}\llll(
\dddd{a''(x)}{\Gamma(\aaaa)t^{1-\aaaa}}+\dfrac{\partial}{\partial t}\dddd{\partial^{1-\aaaa}G(x,t)}{t^{1-\aaaa}}\rrrr)=2b(x).$$
Let 
$$v(x,t)=u(x,t)-L_1(x)t-\dddd{L_2(x)}{2}t^2=u(x,t)-a(x)t-b(x)t^2.$$
The function $v(x,t)$ is a solution of \eqref{VHomogFDE}, where
the function $H(x,t)$ is computed from $G(x,t)$ and the functions $L_1(x)$ and $L_2(x)$ with \eqref{H}
$$H(x,t)=\dddd{1}{2}\Gamma(3+\aaaa)c(x) t^2-c''(x) t^{2+\aaaa}$$
the fractional sub-diffusion equation \eqref{VHomogFDE} has solution
$$v(x,t)=c(x) t^{2+\aaaa}.$$
When $c(x)=2x^2(1-x)$ we obtain the following fractional diffusion equation
	\begin{equation}\label{V2HomogFDE}
	\left\{
	\begin{array}{l l}
\dfrac{\partial^\alpha v(x,t)}{\partial t^\alpha}=\dfrac{\partial^2 v(x,t)}{\partial x^2}+\Gamma(3+\aaaa)(1-x)x^2 t^2-4(3x-1) t^{2+\aaaa},& \\
	v(0,t)=v(1,t)=0, v(x,0)=0.&  \\
	\end{array} 
		\right . 
	\end{equation}
	 Equation \eqref{V2HomogFDE} has  solution
	$v(x,t)=2x^2(1-x)t^{2+\alpha}.$

 When $h=\tttt=0.1$ the error of difference approximation \eqref{FiniteDifferenceScheme2} for equation \eqref{V2HomogFDE} at time $t=1$ is  $0.00097075$. When $h= 0.1\& \tttt=0.05$ the     error is $0.000243735$.  Experimental results for the maximum error and order of approximation \eqref{FiniteDifferenceScheme2} at time  $t=1$ are given in Figure 5, Table 7 and Table 8.
\begin{figure}
\centering
\begin{minipage}{.45\textwidth}
  \centering
  \includegraphics[width=.95\linewidth]{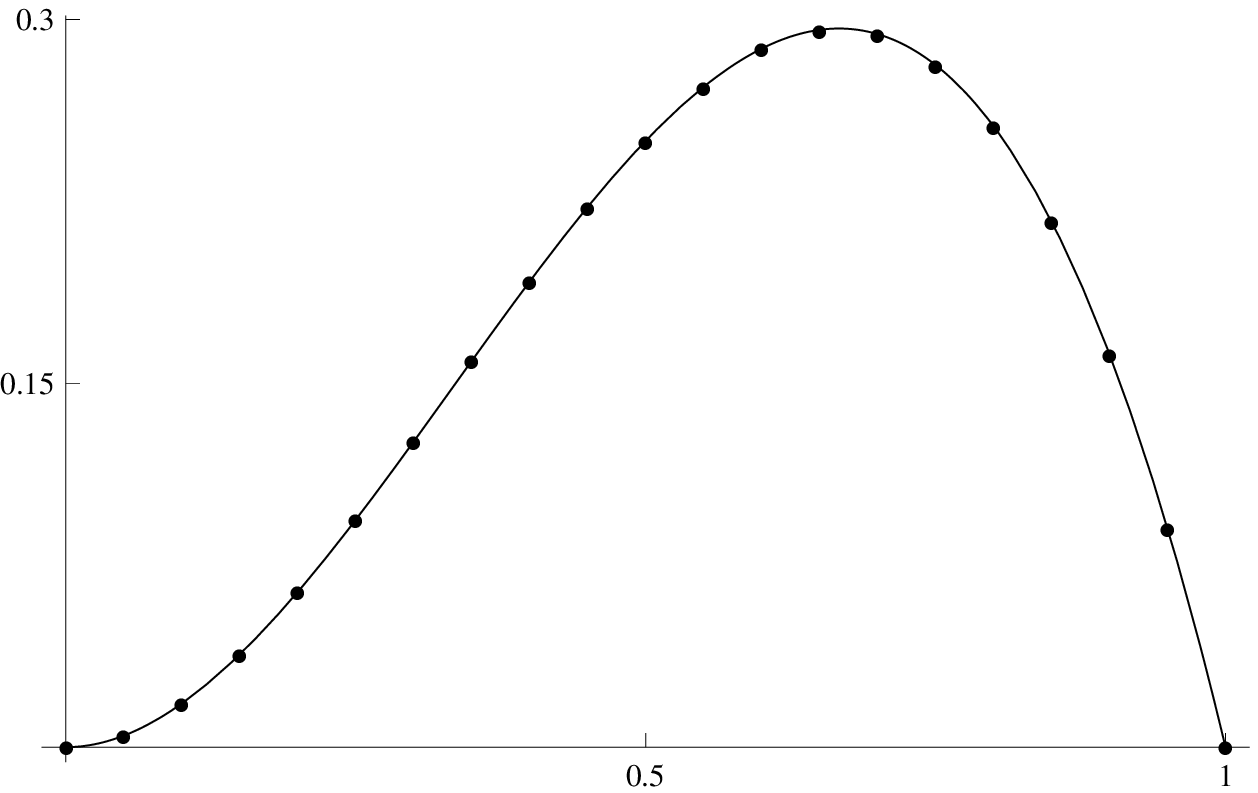}
  \label{fig:test3}
\end{minipage}%
\hspace{0.5cm}
\begin{minipage}{.45\textwidth}
  \centering
  \includegraphics[width=.95\linewidth]{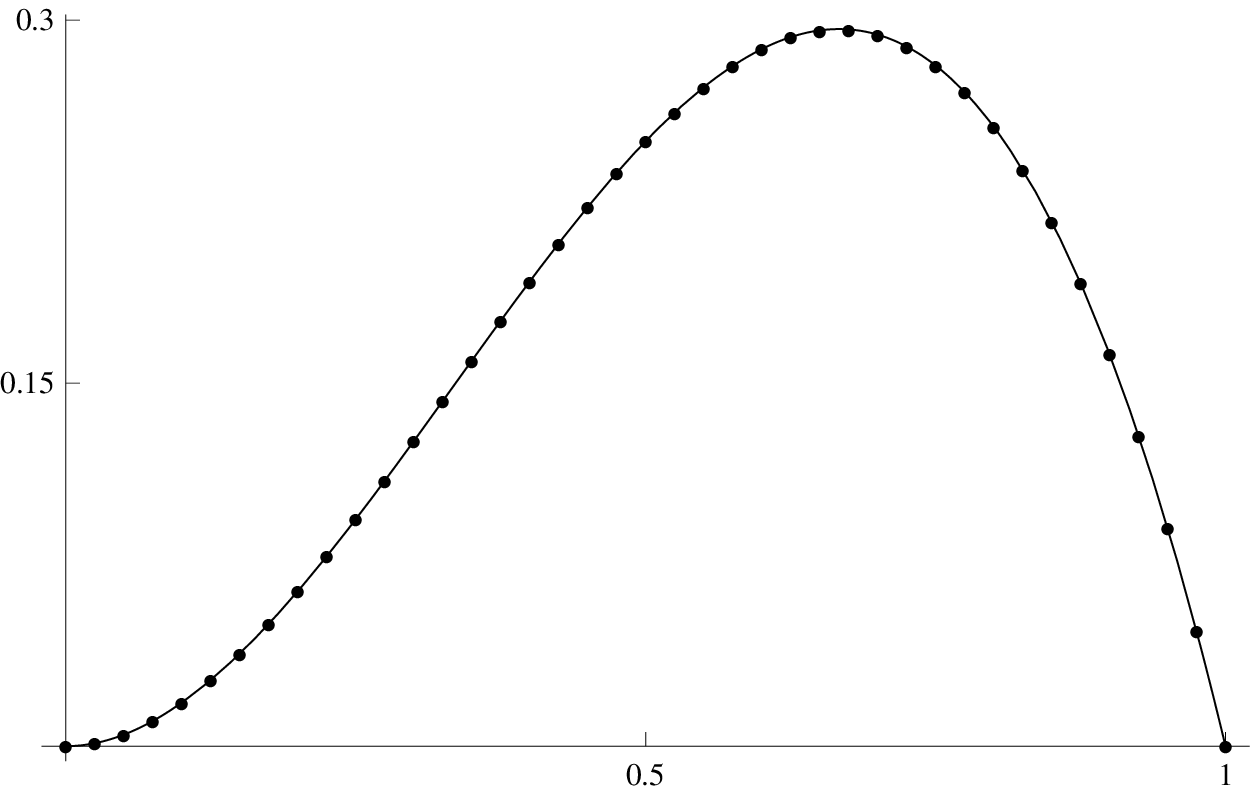}
  \label{fig:test4}
\end{minipage}
 \captionof{figure}{Graphs of the solution of equation \eqref{V2HomogFDE} and  approximation \eqref{FiniteDifferenceScheme2} for $\alpha=1/2$ and $h=\tttt=0.05$ (left)  and $h=0.025,\tttt=0.0125$  at time $t=1$.}
\end{figure}
	\begin{table}
	\caption{Maximum error and order of approximation \eqref{FiniteDifferenceScheme2}  for equation \eqref{V2HomogFDE} with $\alpha=1/2$ and $h=\tttt$ at time $t=1$.}
	\centering
  \begin{tabular}{ l  l l   l c }
    \hline \hline
    $h$  &$\tttt$  & $Error$                 & $Ratio$ & $Order$  
		\\ \hline \hline
$0.05$     &$0.05$     & $0.000244392$                 & $3.97210$   & $1.98990$   \\ 
$0.025$    &$0.025$    & $0.000061436$                 & $3.97802$   & $1.99205$   \\ 
$0.0125$   &$0.0125$   & $0.000015376$         & $3.99552 $  & $1.99838$   \\ 
$0.00625$  &$0.00625$  & $3.846\times10^{-6}$         & $3.99778$   & $1.99920$   \\ 
$0.003125$ &$0.003125$ & $9.619\times10^{-7}$         & $3.99865$   & $1.99951$   \\ \hline

  \end{tabular}
	\end{table}
		\begin{table}
	\caption{Maximum error and order of approximation \eqref{FiniteDifferenceScheme2}  for equation \eqref{V2HomogFDE} with $\alpha=1/2$ and $h=2\tttt$ at time $t=1$.}
	\centering
  \begin{tabular}{ l l l   l c }
    \hline \hline
    $h$  &  $\tttt$  & $Error$                 & $Ratio$ & $Order$  
		\\ \hline \hline
$0.05$     &$0.025$     & $0.000061232$               & $3.98055$   & $1.99297$   \\ 
$0.025$    &$0.0125$    & $0.000015376$               & $3.98235$   & $1.99362$   \\ 
$0.0125$   &$0.00625$   & $3.846\times 10^{-6}$       & $3.99770 $  & $ 1.99917$   \\ 
$0.00625$  &$0.003125$  & $9.618\times 10^{-7} $      & $3.99888$   & $1.99959$   \\ 
$0.003125$ &$0.0015625$ & $2.405\times 10^{-7}$       & $3.99920$   & $1.99971$   \\ \hline     
  \end{tabular}
	\end{table}
\subsection{Numerical Analysis}
Let $E_m=V_m-\wwww{V}_m$ be the error vectors for difference approximations \eqref{FiniteDifferenceScheme} or \eqref{FiniteDifferenceScheme2} at time $t_m=m\tttt$. The vectors $E_m$ are computed recursively  with $E_0=0$ and
$$P E_m=\aaaa Q E_{m-1}+\sum_{k=2}^{m-1} g_k^{(\alpha)} E_{m-k}+K_m,$$
where 
\begin{align}\label{matrices}
P=I+\llll(1-\frac{\aaaa}{2}\rrrr)\eta A, \quad Q=I-\dfrac{\eta}{2} A
\end{align}
and $K_m$ are the vectors of truncation errors at time $t=t_m$.
Define $S=P^{-1}$ and $R=SQ$. Then
\begin{align}\label{error}
E_m=\aaaa R E_{m-1}+\sum_{k=2}^{m-1} g_k^{(\alpha)} SE_{m-k}+S K_m.
\end{align}
In Theorem \ref{t29} we show that the vectors $E_m$ converge to zero with second order accuracy with respect to  $h$ and $\tttt$.

Let $B=(b_{nm})$ be a square matrix of order $N-1$. 
The maximum (infinity) norm  of $B$ is defined as
$$\llll\| B \rrrr\|=\max_{1\leq n\leq N-1} \sum_{m=1}^{N-1}|b_{nm}|$$
The vector and matrix norms satisfy  
\begin{equation}\label{norms}
\llll\| Bx\rrrr\|\leq \llll\| B\rrrr\|\llll\| x\rrrr\|.
\end{equation}
Let $\mu_1,\cdots,\mu_{N-1}$ be the  eigenvalues of  $B$.
The spectral radius of $B$ is the maximum of the absolute values of its eigenvalues.
$$\rho(B) = \max_{1\leq n\leq N-1}|\mu_n|.$$
 The matrices $P,Q,R$ and $S$ are symmetric and commute, because the matrix $A$ is symmetric and definition \eqref{matrices}. The matrix $P$ is a diagonally dominant M-matrix. Then the matrix $S=P^{-1}$ is  positive  and $\llll\| S\rrrr\|\leq 1$. 
The matrix $A$ has eigenvalues
$$\lambda_k=4 \sin^2 \llll( \dddd{k \pi}{N}\rrrr), \quad (k=1,\cdots,N-1).$$
The matrix $P$ has eigenvalues $1+\llll(1-\frac{\aaaa}{2}\rrrr) \eta \lambda_k$, and the eigenvalues of the matrix $P^{-1}$ are
 $$\llll(1+\llll(1-\frac{\aaaa}{2}\rrrr) \eta \lambda_k\rrrr)^{-1}.$$
Now we show that  $R$ is a convergent matrix.
\begin{lem}
$$\rho( R )<1.$$
\end{lem} 
\begin{proof}
The matrix $R=P^{-1}Q$ has eigenvalues,
$$\dfrac{1- \frac{\eta \lambda_k}{2}}{1+\llll(1-\frac{\aaaa}{2}\rrrr) \eta \lambda_k}.$$
Then
$$\llll|\dfrac{1- \frac{\eta \lambda_k}{2}}{1+\llll(1-\frac{\aaaa}{2}\rrrr) \eta \lambda_k}\rrrr|<1,\quad \llll|1- \frac{\eta \lambda_k}{2}\rrrr|<\llll|1+\llll(1-\frac{\aaaa}{2}\rrrr) \eta \lambda_k\rrrr|, $$
$$\llll(1- \frac{\eta \lambda_k}{2}\rrrr)^2<\llll(1+\llll(1-\frac{\aaaa}{2}\rrrr) \eta \lambda_k\rrrr)^2, $$
$$1-\eta \lambda_k+\dfrac{\eta^2 (\lambda_k)^2}{4}<1+2\llll(1-\frac{\aaaa}{2}\rrrr) \eta \lambda_k+\llll(1-\frac{\aaaa}{2}\rrrr)^2\eta^2(\lambda_k)^2,$$
$$-\llll(1+2\llll(1-\frac{\aaaa}{2}\rrrr)\rrrr)\eta \lambda_k<\llll(\llll(1-\frac{\aaaa}{2}\rrrr)^2-\dfrac{1}{4}\rrrr)\eta^2(\lambda_k)^2,$$
$$-\llll(3-\aaaa \rrrr)\eta \lambda_k<\dfrac{1}{4}(1-\aaaa)(3-\aaaa)\eta^2(\lambda_k)^2.$$
The above inequality holds because the left-hand side is negative and the right-hand side is positive.
\end{proof}
The norm and the spectral radius of the matrix $R$ satisfy
$$\rho(R)<\llll\| R\rrrr\|.$$
While the norm of $S$ is smaller than one, the norm of the matrix $R$ may be greater than one. In the proof of Lemma \ref{l28} we use the following property of convergent matrices.
\begin{lem} \label{l20} There exists a positive integer $J$ such that
$$\llll\|R^k S^{J-k}\rrrr\|<1$$
for all $0\leq k\leq J$.
\end{lem}
\begin{proof} The matrices $R$ and $S$ are convergent matrices. Then 
$$ \lim_{k\rightarrow \infty} \llll\|R^k \rrrr\|=
\lim_{k\rightarrow \infty} \llll\|S^k \rrrr\|=0.$$
The sequence $\{\llll\|R^k \rrrr\|\}_{k=0}^\infty$ is bounded. Let $C_R$ be a positive constant such that
$$\llll\|R^k \rrrr\|<C_R,\quad (k=1,2,\dots)$$
and
\begin{align}\label{CR}
C_R>\max\llll\{K,\frac{5K}{(1-\aaaa)2^\aaaa}\rrrr\}.
\end{align}
The number $C_R$ is greater than one, because $K>1$.
Let  $J'$ be a positive integer such that
$$\llll\|R^k \rrrr\|<1,\quad \llll\|S^k \rrrr\|<\dddd{1}{C_R},\quad (k\geq J')$$
Choose $J>2J'$. When $k\geq J' $ we have
$$\llll\|R^k S^{J-k}\rrrr\|\leq\llll\|R^k \rrrr\|\llll\|S^{J-k}\rrrr\|\leq \llll\|R^k \rrrr\|\llll\|S\rrrr\|^{J-k}<1.
$$
If $k < J' $ then $J-k > J' $ and
$$\llll\|R^k S^{J-k}\rrrr\|\leq\llll\|R^k \rrrr\|\llll\|S^{J-k}\rrrr\|<
C_R. \dddd{1}{C_R}=1.
$$
\end{proof}
In addition we require that the number $J$ is large enough such that the following inequality is satisfied
\begin{align}\label{JJ}
(C_R)^J>(\aaaa J C_R+1)K.
\end{align}
Such number exists because $C_R>K>1$ and the exponential function $(C_R)^J$ grows faster than the linear function $(\aaaa J C_R+1)K$. We use properties  \eqref{CR} and \eqref{JJ} of the numbers $C_R$ and $J$ in Corollary \ref{c25} and Lemma \ref{l26}. Denote
$$\Phi_m^{(\aaaa)}=\sum_{k=2}^{m-1} g_k^{(\alpha)} S E_{m-k}.$$
The vectors $E_m$ and $E_{m-1}$ are computed recursively as
\begin{align}\label{equal1}
E_m=\aaaa R E_{m-1}+\Phi_m^{(\aaaa)}+S K_{m},
\end{align}
$$E_{m-1}=\aaaa R E_{m-2}+\Phi_{m-1}^{(\aaaa)}+S K_{m-1}.$$
Then
$$E_m=\aaaa R\llll(\aaaa R E_{m-2}+\Phi_{m-1}^{(\aaaa)}+S K_{m-1}\rrrr)+\Phi_m^{(\aaaa)}+S K_m,$$
\begin{align}\label{equal2}
E_m=\aaaa^2 R^2 E_{m-2}+  \sum_{k=2}^{m-2} \aaaa g_k^{(\alpha)} R S E_{m-k-1}+\Phi_m^{(\aaaa)}+\aaaa R S K_{m-1}+S K_m.
\end{align}
In Lemma \ref{l21} we define the numbers $\bbbb^{(m)}_{n,k,i}$ and the vectors $A_{m,n}$ recursively with \eqref{F1}, \eqref{F2} and \eqref{approx_m}. The boundary values of $\bbbb^{(m)}_{n,k,i}$ and  $A_{m,n}$ are
	\begin{equation*}
	\left\{
	\begin{array}{l l}
	\bbbb^{(m)}_{n,0,m-n}=\aaaa^n,&m\geq 1,  \\
	\bbbb^{(m)}_{n,0,i}=0,& i\neq m-n, \\
	\bbbb^{(m)}_{n,k,i}=0,& k<0,k\geq n,
	\end{array} \label{OrdinaryFractionalDiffEqn}
		\right . 
	\end{equation*}
$$A_{m,0}=S K_m,\quad A_{m,1}=\aaaa R S K_{m-1}+S K_m.$$
\begin{lem} \label{l21} There exist positive numbers $\bbbb^{(m)}_{n,k,i}$ such that the error vector $E_m$ can be represented as
\begin{align}\label{F1}
E_m=\aaaa^n R^n E_{m-n}+\sum_{k=1}^{n-1}\sum_{i=1}^{m-n-k}
\bbbb^{(m)}_{n,k,i}R^{n-k}S^k E_i+\Phi_m^{(\aaaa)}+A_{m,n}.
\end{align}
The numbers $\bbbb^{(m)}_{n,k,i}$  and the vectors $A_{m,n}$ are computed recursively as
\begin{align}\label{F2}
\bbbb^{(m)}_{n+1,k,i}=\aaaa \bbbb^{(m)}_{n,k,i+1}+\sum_{j=2}^{m-n-k-i+1} g_j ^{(\aaaa)}
\bbbb^{(m)}_{n,k-1,j+i},
\end{align}
\begin{align} \label{approx_m}
A_{m,n+1}=A_{m,n}+\aaaa^{n}R^n S K_{m-n}+\sum_{k=1}^{n-1}\sum_{i=1}^{m-n-k}\bbbb^{(m)}_{n,k,i}R^{n-k}S^{k+1}K_i.
\end{align}
\end{lem}
\begin{proof} We prove that \eqref{F2} and \eqref{approx_m} hold by induction on $n$. From the definition of $\bbbb^{(m)}_{n,k,i},A_{m,n}$ and formulas \eqref{equal1} and \eqref{equal2} we have that \eqref{F2} and \eqref{approx_m} hold for $n=1$ and $n=2$. Suppose that \eqref{F2} and \eqref{approx_m} hold for all $n\leq \oooo{n}$.
\begin{align*}
E_m=\aaaa^{\oooo{n}} R^{\oooo{n}} E_{m-\oooo{n}}+\sum_{k=1}^{\oooo{n}-1}\sum_{i=1}^{m-n-k}
\bbbb^{(m)}_{\oooo{n},k,i}R^{\oooo{n}-k}S^k E_i+\Phi_m^{(\aaaa)}+A_{m,\oooo{n}}.
\end{align*}
By substituting the vectors $E_1,E_2,\cdots,E_{m-\oooo{n}}$ with \eqref{error} we get
\begin{align}\label{iter} \nonumber
E_m=& \aaaa^{\oooo{n}} R^{\oooo{n}} \llll(\aaaa R E_{m-\oooo{n}-1}+\sum_{k=2}^{m-\oooo{n}-1} g_k^{(\alpha)} S E_{m-\oooo{n}-k}+SK_{m-\oooo{n}} \rrrr)+\Phi_m^{(\aaaa)}+A_{m,\oooo{n}}\\
&+\sum_{k=1}^{\oooo{n}-1}\sum_{i=1}^{m-\oooo{n}-k}\bbbb_{\oooo{n},k,i}^{(m)}R^{\oooo{n}-k}S^k
\llll(\aaaa R E_{i-1}+\sum_{k=2}^{i-1} g_k^{(\alpha)} S E_{i-k}+S K_i \rrrr).
\end{align}
The formula for recursive computation \eqref{approx_m} of the vectors $A_{m,\oooo{n}}$ of approximation errors is obtained from \eqref{iter} as the sum of the approximation errors.
The coefficient of $R^{\oooo{n}+1-k}S^k$ in formula \eqref{F1} with $n=\oooo{n}+1$
$$\sum_{i=1}^{m-\oooo{n}-k-1} \bbbb^{(m)}_{\oooo{n}+1,k,i}E_i$$
is determined from the coefficients of $R^{\oooo{n}-k}S^k$ and $R^{\oooo{n}+1-k}S^{k-1}$ in \eqref{F1} with $n=\oooo{n}$. The coefficient of  $R^{\oooo{n}-k}S^k$ is
$$\sum_{i=1}^{m-\oooo{n}-k} \bbbb^{(m)}_{\oooo{n},k,i}E_i.$$
After one iteration the coefficient becomes
$$\sum_{i=1}^{m-\oooo{n}-k} \aaaa \bbbb^{(m)}_{\oooo{n},k,i}E_{i-1}=\sum_{i=1}^{m-\oooo{n}-k-1} \aaaa \bbbb^{(m)}_{\oooo{n},k,i+1}E_{i}.$$
Similarly, the coefficient of  $R^{\oooo{n}-k+1}S^{k-1}$ is initially
$$\sum_{i=1}^{m-\oooo{n}-k+1} \bbbb^{(m)}_{\oooo{n},k-1,i}E_i.$$
After one iteration it becomes
$$\sum_{i=1}^{m-\oooo{n}-k+1} \bbbb^{(m)}_{\oooo{n},k-1,i} \sum_{j=2}^{i-1}g_j^{(\aaaa)}E_{i-j}= \sum_{i=3}^{m-\oooo{n}-k+1} \sum_{j=2}^{i-1} g_j^{(\aaaa)}\bbbb^{(m)}_{\oooo{n},k-1,i} E_{i-j}.$$
By substituting $l=i-j$ we obtain
\begin{equation}\label{change}
\sum_{i=3}^{m-\oooo{n}-k+1} \sum_{j=2}^{i-1} g_j^{(\aaaa)} \bbbb^{(m)}_{\oooo{n},k-1,i} E_{i-j}
=\sum_{l=1}^{m-\oooo{n}-k-1}\sum_{j=2}^{m-\oooo{n}-k-l+1}g_j^{(\aaaa)}\bbbb^{(m)}_{\oooo{n},k-1,j+l} E_l.
\end{equation}
Then
$$\sum_{i=1}^{m-\oooo{n}-k-1} \bbbb^{(m)}_{\oooo{n}+1,k,i}E_i=
\sum_{i=1}^{m-\oooo{n}-k-1} \aaaa \bbbb^{(m)}_{\oooo{n},k,i+1}E_{i}+\sum_{i=1}^{m-\oooo{n}-k-1}\sum_{j=2}^{m-\oooo{n}-k-i+1}g_j^{(\aaaa)}\bbbb^{(m)}_{\oooo{n},k-1,j+i} E_i,$$
$$\sum_{i=1}^{m-\oooo{n}-k-1} \bbbb^{(m)}_{\oooo{n}+1,k,i}E_i=
\sum_{i=1}^{m-\oooo{n}-k-1}\llll( \aaaa \bbbb^{(m)}_{\oooo{n},k,i+1}+\sum_{j=2}^{m-\oooo{n}-k-i+1}g_j^{(\aaaa)}\bbbb^{(m)}_{\oooo{n},k-1,j+i} \rrrr)E_i.$$
The coefficients of $E_i$ are equal. Therefore
$$ \bbbb^{(m)}_{\oooo{n}+1,k,i}=\aaaa \bbbb^{(m)}_{\oooo{n},k,i+1}+\sum_{j=2}^{m-\oooo{n}-k-i+1}g_j^{(\aaaa)}\bbbb^{(m)}_{\oooo{n},k-1,j+i}.$$
We use \eqref{F2} for recursive computation of all coefficients $\bbbb^{(m)}_{n+1,k,i}$ for $k=1,\cdots,n-1$. The formula also hols in the boundary cases $k=0$ and $k=n$. When $k=0$ we have,
$$ \bbbb^{(m)}_{n+1,0,i}=\aaaa \bbbb^{(m)}_{n,0,i+1}+\sum_{j=2}^{m-n-k-i+1}g_j^{(\aaaa)}\bbbb^{(m)}_{n,-1,j+i}=\aaaa \bbbb^{(m)}_{n,0,i+1} $$
because $\bbbb^{(m)}_{n,-1,j+i}=0$. Then 
$$\bbbb^{(m)}_{n+1,0,m-n-1}=\aaaa \bbbb^{(m)}_{n,0,m-n}=\aaaa^{n+1}.$$
When $k=n$,
$$ \bbbb^{(m)}_{n+1,n,i}=\aaaa \bbbb^{(m)}_{n,n,i+1}+\sum_{j=2}^{m-2n-i+1}g_j^{(\aaaa)}\bbbb^{(m)}_{n,n-1,j+i}=\sum_{j=2}^{m-2n-k-i+1}g_j^{(\aaaa)}\bbbb^{(m)}_{n,n-1,j+i}. $$
\end{proof}
A more convenient way to write formulas \eqref{F1} and \eqref{approx_m} is
\begin{align}\label{FF1}
E_m=\sum_{k=0}^{n-1}\sum_{i=1}^{m-n-k}
\bbbb^{(m)}_{n,k,i}R^{n-k}S^k E_i+\Phi_m^{(\aaaa)}+A_{m,n},
\end{align}
$$A_{m,n+1}=A_{m,n}+\sum_{k=0}^{n-1}\sum_{i=1}^{m-n-k}\bbbb^{(m)}_{n,k,i} R^{n-k}S^{k+1} K_i.$$
Denote,
$$\wwww{\bbbb^{(m)}}_n=\sum_{k=0}^{n-1}\sum_{i=1}^{m-n-k} \bbbb^{(m)}_{n,k,i}.$$
\begin{cor} The sequence $\wwww{\bbbb^{(m)}}_n$ is decreasing. 
\end{cor}
\begin{proof}
\begin{align*}
\wwww{\bbbb^{(m)}}_{n+1}=&\sum_{k=0}^n \sum_{i=1}^{m-n-k-1}\bbbb^{(m)}_{n+1,k,i}=\\
&\sum_{k=0}^n\llll( \aaaa  \sum_{i=1}^{m-n-k-1}\bbbb^{(m)}_{n,k,i+1}+
\sum_{i=1}^{m-n-k-1}\sum_{j=2}^{m-n-k-i+1} 
g_j^{(\aaaa)}\bbbb^{(m)}_{n,k-1,j+i}\rrrr).
\end{align*}
Substitute $l=i+j$.
$$\wwww{\bbbb^{(m)}}_{n+1}=\sum_{k=0}^n\llll(\aaaa \sum_{i=1}^{m-n-k-1} \bbbb^{(m)}_{n,k,i+1} +
\sum_{i=3}^{m-n-k}\sum_{l=2}^{i-1}\bbbb^{(m)}_{n,k-1,i}g_l^{(\aaaa)} \rrrr).$$
We have that
$$\sum_{l=2}^{i-1}g_l^{(\aaaa)}<\sum_{l=2}^{\infty}g_l^{(\aaaa)}=1-\aaaa.$$
Then
$$\wwww{\bbbb^{(m)}}_{n+1}<\sum_{k=0}^n\llll(\aaaa \sum_{i=2}^{m-n-k} \bbbb^{(m)}_{n,k,i} +(1-\aaaa)
\sum_{i=3}^{m-n-k}\bbbb^{(m)}_{n,k,i} \rrrr),$$
$$\wwww{\bbbb^{(m)}}_{n+1}<\sum_{k=0}^n \sum_{i=1}^{m-n-k} \bbbb^{(m)}_{n,k,i}=\wwww{\bbbb^{(m)}}_{n}.$$
\end{proof}
The value of $\wwww{\bbbb^{(m)}}_{1}$ is
$$\wwww{\bbbb^{(m)}}_{1}=\sum_{i=1}^{m-1}\bbbb_{1,0,i}=\bbbb_{1,0,m-1}=\aaaa.$$
\begin{cor} (Estimate for sums of coefficients of \eqref{FF1}) 
$$\sum_{k=0}^{n-1}\sum_{i=1}^{m-n-k} \bbbb^{(m)}_{n,k,i}<\aaaa.$$
\end{cor}
\begin{proof}  The sequence $\llll\{\wwww{\bbbb^{(m)}}_{n}\rrrr\}$ is decreasing. Then 
$$\bbbb^{(m)}_{n}<\wwww{\bbbb}^{(m)}_1=\aaaa.$$
\end{proof}
\begin{lem} (Estimate for the norm of $A_{m,n}$)
\begin{equation}\label{L1}
\llll\|A_{m,n}\rrrr\|\leq (n\aaaa C_R+1)K\tttt^\aaaa \llll(\tttt^2+h^2\rrrr).
\end{equation}
\end{lem}
\begin{proof} When $n=1$ we have
$$A_{m,1}=\aaaa R S K_{m-1}+S K_m,$$
$$\llll\|A_{m,1}\rrrr\|\leq \aaaa \llll\|R \rrrr\| \llll\|S \rrrr\| \llll\|K_{m-1} \rrrr\|+\llll\|S \rrrr\| \llll\|K_m \rrrr\| \leq (\aaaa C_R+1)K\tttt^\aaaa \llll(\tttt^2+h^2\rrrr).$$
We prove \eqref{L1} by induction on $n$. Suppose that \eqref{L1} holds for  $n\leq \oooo{n}$. The vectors $A_{m,\oooo{n}}$ are computed recursively with
$$A_{m,\oooo{n}+1}=A_{m,\oooo{n}}+\sum_{k=0}^{\oooo{n}-1}\sum_{i=1}^{m-\oooo{n}-k}\bbbb^{(m)}_{\oooo{n},k,i} R^{\oooo{n}-k}S^{k+1} K_i.$$
Then
$$\llll\|A_{m,\oooo{n}+1}\rrrr\|\leq \llll\|A_{m,\oooo{n}}\rrrr\|+\sum_{k=0}^{\oooo{n}-1}\sum_{i=1}^{m-\oooo{n}-k}\bbbb^{(m)}_{\oooo{n},k,i} \llll\| R^{\oooo{n}-k}\rrrr\|\llll\|S^{k+1} \rrrr\|\llll\|K_i \rrrr\|,$$
$$\llll\|A_{m,\oooo{n}+1}\rrrr\|\leq \llll\|A_{m,\oooo{n}}\rrrr\|+
K C_R \tttt^\aaaa \llll(\tttt^2+h^2\rrrr)\sum_{k=0}^{\oooo{n}-1}\sum_{i=1}^{m-\oooo{n}-k}\bbbb^{(m)}_{\oooo{n},k,i}.$$
From Corollary 23 and the induction hypothesis
$$\llll\|A_{m,\oooo{n}+1}\rrrr\|\leq \llll\|A_{m,\oooo{n}}\rrrr\|+
\aaaa K C_R \tttt^\aaaa \llll(\tttt^2+h^2\rrrr),$$
$$\llll\|A_{m,\oooo{n}+1}\rrrr\|\leq (\oooo{n}\aaaa C_R+1)K\tttt^\aaaa \llll(\tttt^2+h^2\rrrr)+\aaaa K C_R \tttt^\aaaa \llll(\tttt^2+h^2\rrrr),$$
$$\llll\|A_{m,\oooo{n}+1}\rrrr\|\leq ((\oooo{n}+1)\aaaa C_R+1)K\tttt^\aaaa \llll(\tttt^2+h^2\rrrr).$$
\end{proof}
By setting $n=J$, where $J$ is the number determined in Lemma \ref{l20} and \eqref{JJ}, and combining the results from Corollary 23 and Corollary 24  we obtain.
\begin{cor} \label{c25}The vectors $E_m$ are computed recursively with $E_0=0$ and
\begin{align}\label{recEm}
E_m=\sum_{k=0}^{J-1}\sum_{i=1}^{m-J-k}
\bbbb^{(m)}_{J,k,i}R^{J-k}S^k E_i+\sum_{k=2}^{m-1}g_k^{(\aaaa)}SE_{m-k}+A_{m,J},
\end{align}
where the numbers $\bbbb^{(m)}_{J,k,i}\geq 0, g_k^{(\aaaa)}>0$  and the vectors $A_{m,J}$ satisfy
\begin{align}\label{est}
 \sum_{k=0}^{J-1}\sum_{i=1}^{m-J-k}\bbbb^{(m)}_{J,k,i}+
\sum_{k=2}^{m-1} g_k^{(\aaaa)}< \sum_{k=1}^{m-1} g_k^{(\aaaa)},
\end{align}
$$\llll\|A_{m,J}\rrrr\|\leq (\aaaa J C_R+1)K \tttt^\aaaa \llll(\tttt^2+h^2\rrrr)< C_R^J \tttt^\aaaa \llll(\tttt^2+h^2\rrrr).$$

\end{cor}
In the next two lemmas we determine  estimates for the error vectors $E_m$.
\begin{lem}\label{l26} Let $m\leq J$. Then
\begin{align}\label{F60}
\llll\|E_m\rrrr\|< (C_R)^m m^\aaaa \tttt^\aaaa \llll(\tttt^2+h^2\rrrr).
\end{align}
\end{lem}
\begin{proof} Induction on $m$:
$$E_1=\aaaa R E_0+S K_1=S K_1,$$
$$\llll\| E_1\rrrr\|\leq \llll\| S K_1 \rrrr\| \leq \llll\| S \rrrr\|
\llll\|  K_1 \rrrr\|\leq \llll\|  K_1 \rrrr\|<
K \tttt^\aaaa \llll(\tttt^2+h^2 \rrrr)<
C_R \tttt^\aaaa \llll(\tttt^2+h^2 \rrrr).$$
Suppose that \eqref{F60} holds for $m<\oooo{m}$.
The vector $E_{\oooo m}$ is computed recursively with 
$$E_{\oooo{m}}=\aaaa R E_{\oooo{m}-1}+\sum_{k=2}^{\oooo{m}-1}g_k^{(\aaaa)} S E_{\oooo{m}-k}+S K_{\oooo{m}}.$$
Then
$$\llll\| E_{\oooo{m}} \rrrr\|\leq \aaaa \llll\| R \rrrr\|\ \llll\| E_{\oooo{m}-1}\rrrr\|+\sum_{k=2}^{\oooo{m}-1} g_k^{(\aaaa)} \llll\| S \rrrr\|\llll\| E_{\oooo{m}-k} \rrrr\|+\llll\| S \rrrr\|\llll\| K_{\oooo{m}} \rrrr\|,$$
$$\llll\| E_{\oooo{m}} \rrrr\|\leq \aaaa C_R\ \llll\| E_{\oooo{m} -1}\rrrr\|+\sum_{k=2}^{\oooo{m}-1} g_k^{(\aaaa)} \llll\| E_{\oooo{m}-k} \rrrr\|+\llll\| K_{\oooo{m}} \rrrr\|.$$
By the induction hypothesis 
$$\llll\| E_{\oooo{m}-k}\rrrr\|<(C_R)^{\oooo{m}-k} (\oooo{m}-k)^\aaaa \tttt^\aaaa \llll(\tttt^2+h^2 \rrrr)\leq
(C_R)^{\oooo{m}-1} \oooo{m}^\aaaa \tttt^\aaaa \llll(\tttt^2+h^2 \rrrr).$$
Then
\begin{align*}
\dddd{\llll\| E_{\oooo{m}} \rrrr\|}{\tttt^\aaaa \llll(\tttt^2+h^2 \rrrr)}<
\aaaa (C_R)^{\oooo{m}} \oooo{m}^\aaaa+(C_R)^{\oooo{m}-1}\oooo{m}^\aaaa \sum_{k=2}^{\oooo{m}-1}g_k^{(\aaaa)}+K,
\end{align*}
\begin{align*}
\dddd{\llll\| E_{\oooo{m}} \rrrr\|}{\tttt^\aaaa \llll(\tttt^2+h^2 \rrrr)}<
\aaaa (C_R)^{\oooo{m}} \oooo{m}^\aaaa+(C_R)^{\oooo{m}-1}\oooo{m}^\aaaa \sum_{k=2}^{\infty}g_k^{(\aaaa)}
-(C_R)^{\oooo{m}-1}\oooo{m}^\aaaa \sum_{k=\oooo{m}}^{\infty}g_k^{(\aaaa)}+K.
\end{align*}
From Lemma \ref{l12},
$$\sum_{k=\oooo{m}}^{\infty}g_k^{(\aaaa)}>\dddd{1-\aaaa}{5}\dddd{2^{\aaaa}}{\oooo{m}^\aaaa}.$$
Then
\begin{align*}
\dddd{\llll\| E_{\oooo{m}} \rrrr\|}{\tttt^\aaaa \llll(\tttt^2+h^2 \rrrr)}<\aaaa (C_R)^{\oooo{m}} \oooo{m}^\aaaa+(1-\aaaa) (C_R)^{\oooo{m}}\oooo{m}^\aaaa -C_R\oooo{m}^\aaaa \dddd{1-\aaaa}{5}\dddd{2^{\aaaa}}{\oooo{m}^\aaaa}+K.
\end{align*}
\begin{align*}
\dddd{\llll\| E_{\oooo{m}} \rrrr\|}
{\tttt^\aaaa \llll(\tttt^2+h^2 \rrrr)}
< (C_R)^{\oooo{m}} \oooo{m}^\aaaa
-
 \dddd{(1-\aaaa)2^{\aaaa}}{5}C_R+K.
\end{align*}
Hence,
$$\llll\| E_{\oooo m} \rrrr\|\leq  (C_R)^{\oooo{m}} \oooo{m}^\aaaa \tttt^\aaaa \llll( \tttt^2+h^2\rrrr).$$
because $C_R>5K/((1-\aaaa)2^{\aaaa})$.
\end{proof}
\begin{cor} Let $m\leq J$. Then
$$\llll\| E_m \rrrr\|\leq  (C_R)^J m^\aaaa\tttt^\aaaa \llll( \tttt^2+h^2\rrrr).$$
\end{cor}
\begin{lem} \label{l28} (Estimate for the vectors $E_m$)
\begin{align}\label{L28}
\llll\| E_m\rrrr\|<Cm^\aaaa \tttt^\aaaa\llll( \tttt^2+h^2 \rrrr),
\end{align}
where
$$C=\max\llll\{(C_R)^J,\dddd{5(C_R)^J}{(1-\aaaa)2^\aaaa}\rrrr\}.$$
\end{lem} 
\begin{proof} Induction on $m$. From Corollary 27 estimate \eqref{L28} holds for $m\leq J$.  Suppose that \eqref{L28} holds for  $m<\oooo{m}$, where $\oooo{m}>J$. The vector $E_{
\oooo{m}}$ is computed recursively with \eqref{recEm}
$$E_{\oooo{m}}=\sum_{k=0}^{J-1}\sum_{i=1}^{\oooo{m}-J-k}
\bbbb^{(\oooo{m})}_{J,k,i}R^{J-k}S^k E_i+\sum_{k=2}^{\oooo{m}-1}g_k^{(\aaaa)}SE_{\oooo{m}-k}+A_{\oooo{m},J}.$$
Then
$$\llll\| E_{\oooo{m}} \rrrr\|\leq \sum_{k=0}^{J-1}\sum_{i=1}^{\oooo{m}-J-k}
\bbbb^{(\oooo{m})}_{J,k,i}\llll\|R^{J-k}S^k\rrrr\| \llll\|E_i\rrrr\|+\sum_{k=2}^{\oooo{m}-1}g_k^{(\aaaa)}\llll\|S\rrrr\|\llll\|E_{\oooo{m}-k}\rrrr\|+\llll\|A_{\oooo{m},J}\rrrr\|.$$
The number $J$ is chosen in Lemma \ref{l20} such that $\llll\|R^{J-k}S^k\rrrr\|<1$. Then
$$\llll\| E_{\oooo{m}} \rrrr\|< \sum_{k=0}^{J-1}\sum_{i=1}^{\oooo{m}-J-k}
\bbbb^{(\oooo{m})}_{J,k,i} \llll\|E_i\rrrr\|+\sum_{k=2}^{\oooo{m}-1}g_k^{(\aaaa)}\llll\|E_{\oooo{m}-k}\rrrr\|+\llll\|A_{\oooo{m},J}\rrrr\|.$$
By the inductive hypothesis 
$$\llll\|E_{\oooo{m}-k}\rrrr\|<
C (\oooo{m}-k)^\aaaa \tttt^\aaaa\llll( \tttt^2+h^2 \rrrr)
<C \oooo{m}^\aaaa \tttt^\aaaa\llll( \tttt^2+h^2 \rrrr).$$
Then
$$\llll\| E_{\oooo{m}} \rrrr\|<C \oooo{m}^\aaaa \tttt^\aaaa \llll(\tttt^2+h^2\rrrr)\llll( \sum_{k=0}^{J-1}\sum_{i=1}^{\oooo{m}-J-k}\bbbb^{(\oooo{m})}_{J,k,i} +\sum_{k=2}^{\oooo{m}-1}g_k^{(\aaaa)}\rrrr)+(C_R)^J \tttt^\aaaa \llll(\tttt^2+h^2\rrrr).$$
From \eqref{est},
$$\llll\| E_{\oooo{m}} \rrrr\|<C \oooo{m}^\aaaa \tttt^\aaaa \llll(\tttt^2+h^2\rrrr)\sum_{k=1}^{\oooo{m}-1}g_k^{(\aaaa)}+(C_R)^J \tttt^\aaaa \llll(\tttt^2+h^2\rrrr),$$
$$\dddd{\llll\| E_{\oooo{m}} \rrrr\|}{\tttt^\aaaa \llll(\tttt^2+h^2\rrrr)}< C\oooo{m}^\aaaa \llll(\sum_{k=1}^{\infty}g_k^{(\aaaa)}- \sum_{k=\oooo{m}}^{\infty}g_k^{(\aaaa)}\rrrr)+(C_R)^J.$$
We have that
$$\sum_{k=1}^{\infty}g_k^{(\aaaa)}=1,\quad \sum_{k=\oooo{m}}^{\infty}g_k^{(\aaaa)}>
\dddd{1-\aaaa}{5}\llll(\dddd{2}{\oooo{m}} \rrrr)^\aaaa.$$
Then
$$\dddd{\llll\| E_{\oooo{m}} \rrrr\|}{\tttt^\aaaa \llll(\tttt^2+h^2\rrrr)}< C\oooo{m}^\aaaa -C\oooo{m}^\aaaa  \dddd{1-\aaaa}{5}\llll(\dddd{2}{\oooo{m}} \rrrr)^\aaaa+(C_R)^J,$$
$$\dddd{\llll\| E_{\oooo{m}} \rrrr\|}{\tttt^\aaaa \llll(\tttt^2+h^2\rrrr)}< C\oooo{m}^\aaaa - \dddd{(1-\aaaa)2^\aaaa}{5}C+(C_R)^J,$$
$$\llll\| E_{\oooo{m}} \rrrr\|<C\oooo{m}^\aaaa 
\tttt^\aaaa\llll( \tttt^2+h^2 \rrrr),$$
because $C>5(C_R)^J/((1-\aaaa)2^\aaaa)$.
\end{proof}
\begin{thm} \label{t29} Difference approximations \eqref{FiniteDifferenceScheme} and \eqref{FiniteDifferenceScheme2} are unconditionally stable and converge to the solution of \eqref{VHomogFDE} with second order accuracy with respect to the space and time variables.
\end{thm}
\begin{proof} The value of $\tttt$ is $\tttt=T/M$. From Lemma 28,
$$\llll\| E_m\rrrr\|<C m^\aaaa 
\tttt^\aaaa\llll( \tttt^2+h^2 \rrrr)\leq C T^\aaaa\llll(\dddd{m}{M}\rrrr)^\aaaa\llll( \tttt^2+h^2 \rrrr),$$
$$\llll\| E_m\rrrr\|< CT^\aaaa\llll( \tttt^2+h^2 \rrrr).$$
for all $m\leq M$.
\end{proof}
\section{Acknowledgements}
I would like to thank Prof. Luben Valkov for useful discussions during the work on this paper.

\end{document}